\documentclass{amsart}
\usepackage[utf8]{inputenc}

\title{Malgrange-Galois groupoid of Painlevé VI equation with parameters}

\author[D. Bl\'azquez-Sanz]{David Bl\'azquez-Sanz}
\author[G. Casale]{Guy Casale}
\author[J. S. Sebasti\'an D\'iaz Arboleda]{Juan Sebastián D\'iaz Arboleda}

\address{David Bl\'azquez-Sanz \& Juan Sebastián D\'iaz Arboleda \hfill\break\indent Universidad Nacional de Colombia - Sede Medell\'in \hfill\break\indent  Facultad de Ciencias 
\hfill\break\indent Escuela de Matem\'aticas \hfill\break\indent  Medell\'in, Colombia}
\email{dblazquezs@unal.edu.co}
\email{jsdiaz@unal.edu.co}

\address{Guy Casale\hfill\break\indent Univ Rennes, CNRS, IRMAR-UMR 6625, F-35000 Rennes, France}
\email{guy.casale@univ-rennes1.fr}



\date{April 2020}

\usepackage{times}
\usepackage{dsfont}
\usepackage{amssymb,hyperref,color,amsthm,verbatim}
\usepackage[mathscr]{euscript}
\usepackage{pb-diagram}
\usepackage[all]{xy}

\usepackage{cleveref}
\usepackage{multirow}
\usepackage{multicol}
\usepackage{xparse}
\usepackage{amsmath}
\usepackage{amsfonts}
\usepackage{amsthm}
\usepackage{xcolor}
\usepackage[normalem]{ulem}
\usepackage{cancel}

\NewDocumentCommand{\tens}{t_}
 {%
  \IfBooleanTF{#1}
   {\tensop}
   {\otimes}%
 }
\NewDocumentCommand{\tensop}{m}
 {%
  \mathbin{\mathop{\otimes}\displaylimits_{#1}}%
 }

\begin{document}

\makeatletter
\newenvironment{dem}{Proof:}{\end{proof}}

\makeatother

\numberwithin{equation}{section}
\theoremstyle{plain}\newtheorem{theorem}{Theorem}[section]

\theoremstyle{plain}\newtheorem*{theo}{Theorem}

\theoremstyle{plain}\newtheorem{proposition}[theorem]{Proposition}
\theoremstyle{plain}\newtheorem{lemma}[theorem]{Lemma}
\theoremstyle{plain}\newtheorem{corollary}[theorem]{Corollary}

\theoremstyle{definition}\newtheorem{definition}[theorem]{Definition}
\theoremstyle{remark}\newtheorem{remark}[theorem]{Remark}

\theoremstyle{definition}\newtheorem{example}[theorem]{Example}

\newcommand{\fn}[5]{
            \begin{array}{cccl}
                #1: & #2 & \longrightarrow &  {#3}\\
                    & #4 & \longmapsto & {#1}(#4)={#5}
            \end{array}}
            
\newcommand{\na}[3]{\left( \mathcal{L}_{#1}\nabla \right)_{#2}{#3}=0}

\maketitle

\begin{abstract}
    The Malgrange-Galois groupoid of Painlevé IV equations is known to be, for very general values of parameters, the pseudogroup of transformations of the phase space preserving a volume form, a time form and the equation. Here we compute the Malgrange-Galois groupoid of Painlev\'e VI family including all parameters as new dependent variables. We conclude it is the pseudogroup of transformations preserving parameter values, the differential of the independent variable, a volume form in the dependent variables and the equation. This implies that a solution of Painlevé VI depending analytically on parameters does not satisfy any new partial differential equation (including derivatives w. r. t. parameters) which is not derived from Painlevé VI. 
\end{abstract}


\section{Introduction}

This article is devoted to the computation of the Malgrange-Galois groupoid of the Painlevé VI family,
\begin{eqnarray*}
u'' = 
\frac{1}{2} \left( \frac{1}{u} + \frac{1}{u-1} + \frac{1}{u-x} \right){u'}^2 -\left( \frac{1}{x} + \frac{1}{x-1} + \frac{1}{u-x} \right){u'} 
\\
+ \frac{ u(u-1)(u-x)}{x^2(x-1)^2} \left( \frac{c^2}{2} - {\frac{a^2}{2}}\frac{x}{u^2} + {\frac{b^2}{2}}\frac{x-1}{(u-1)^2} + \frac{1-e^2}{2} \frac{x(x-1)}{(u-x)^2} \right);
\end{eqnarray*}
including the parameters $a\,,b\,,c\,,\,e$ as variables. In order to make use of the invariant volume, we consider the Painlevé VI family in its Hamiltonian form, as a vector field in $\mathbb C^3\times\mathbb C^4$:
\begin{align*}
\vec X &= \frac{\partial}{\partial x} + \frac{\partial H}{\partial q} \frac{\partial}{\partial p}  -\frac{\partial H}{\partial p}\frac{\partial}{\partial q} \tag{$\rm HP_{VI}$}\\
H &= \frac{1}{x(x-1)} \Big[ p(p-1)(p-x)q^2 - \Big( a(p-1)(p-x) + \\
& b p(p-x) +(e-1)p(p-1)\Big) q  +\frac{1}{4}\big( (a+b+e-1)^2 -c^2\big) (p-x) \Big].
\end{align*}
Our main result is that the Malgrange-Galois groupoid of the Painlevé VI family is the algebraic pseudogroup of transformations of $\mathbb C^7$ fixing the values of the parameters (here $\bar\pi : \mathbb C^3 \times \mathbb C^4 \to \mathbb C^4$ stands for the cartesian projection in the parameter space),
inducing a translation on the $x$-axis, preserving the volume form $dp\wedge dq \wedge dx$ and preserving $\vec X$.  \\

{\noindent \bf Theorem \ref{th:Final_Mal_groupoid_PVI}. }{\it 
The Malgrange-Galois groupoid of Painlev\'e VI equation is given by
\begin{align}
\mathrm{Mal}(\vec X)  
&= \Big\{ \phi \mid
 \bar\pi \circ \phi = \bar\pi; \;\; \phi^*(dx) = dx; \;\; \phi_*(\vec X)= \vec X;\nonumber
\\
& \qquad \qquad \phi^*(dp\wedge dq) \equiv dp\wedge dq \mod da,db,dc,de,dx
\Big\}.
\end{align}

}

As a consequence we have a functional hyper-transcendence result concerning parameter dependent solutions of PVI. A solution depending analytically on the parameters cannot satisfy any partial differential equation involving derivatives with respect to the independent variable and the parameters, except those equations that are derived from Painlevé VI.\\

{\noindent \bf Corollary \ref{co:par}. }{\it 
If $y(x,a,b,c,e)$ is a parameter dependent solution of the sixth Painlev\'e equation then its annihilator in $\mathcal O_{J(\mathbb C^5, \mathbb C)}$ is the $\partial$-ideal generated by the sixth Painlev\'e equation.}\\

Let us us briefly recall the historical development of differential Galois theory. The ideas of E. Galois on algebraic equations were extended by E. Picard \cite{picard1887equations} and E. Vessiot \cite{vessiot1904theorie} to linear differential equations. In his Ph.D. thesis \cite{drach1898essai} J. Drach attempted to define a Galois group for nonlinear differential equations. In his article, the group-like objet is an algebraic pseudogroup. E. Vessiot spent years to correct Drach's mistake. He succeded is \cite{vessiot1946generale} but his work was forgotten. Later, E. Kolchin \cite{kolchin1973} developed Picard-Vessiot theory in the framework of differential field extensions and algebraic groups. To an order $n$ linear differential equation with coefficients in the differential field $\left( \mathbb C(x), \frac{\partial}{\partial x} \right)$ is associated an algebraic subgroup of ${\rm GL}_n(\mathbb C)$.  This group measures algebraic relations between a basis of solutions and theirs derivatives.\\

At the turn of the 21st century, H. Umemura \cite{umemura1996differential} and B. Malgrange \cite{malgrange2001} independently  defined a group-like objet associated to a (nonlinear) differential equation. In Umemura's work this object is a Lie-Ritt functor, a special case of infinitesimal group. In Malgrange's work the object is an algebraic pseudogroup ($\mathcal D$-groupoid) that we call Malgrange-Galois groupoid. Simultaneously, an important refinement of the Galois group for linear differential equations was introduced by P. Landesman \cite{landesman2008generalized} and by P. J. Cassidy and M. F. Singer \cite{cassidy2005galois}. The ``parameterized'' Galois group of an order $n$ linear differential equation with coefficients\footnote{The extension of parameterized Picard-Vessiot theory that allows coefficients in finitely generated fields can be consulted in \cite{sanz2018differential}.} in $\left(\mathbb C(t,x),\frac{\partial}{\partial x}\right)$ ($t$ is a set of parameters),  is an algebraic differential subgroup of   
$\left( {\rm GL}_n(\mathbb C(t)), \frac{\partial}{\partial t}\right)$. This ``group'' is a special case of intransitive pseudogroup and it measures the differential relations between a basis of solution when one derives them with respect to $x$ and to $t$. \\

In Malgrange's nonlinear differential Galois theory there is no distinction between parameters and variables. Therefore, in the complex case, the classical and parameterized Picard-Vessiot theory can be submersed in the nonlinear differential Galois theory. The corresponding differential Galois groups are recovered as the diagonal part of the Malgrange-Galois groupoid. Thus, the Malgrange-Galois groupoid of Painlevé VI family is the nonlinear analog of the differential algebraic Galois group in parameterized Picard-Vessiot theory. Nonlinear monodromy of Painlev\'e VI gives elements of Malgrange-Galois groupoid. The nonlinear monodromy seen as a selfmap on $\mathbb C ^2 \times \mathbb C^4$ preserving the projection on the parameters space $\mathbb C^4$ does not satisfy a common algebraic partial differential equation with coefficients in $\left(\mathbb C (p,q,a,b,c,e), \frac{\partial }{\partial p}, \frac{\partial }{\partial q}, \frac{\partial }{\partial a}, \frac{\partial }{\partial b}, \frac{\partial }{\partial c}, \frac{\partial }{\partial e}\right)$ which is not a consequence of the invariance of $dp\wedge dq \wedge da \wedge db \wedge dc \wedge de$.  \\

The first attempt to compute the Galois groupoid of a Painlev\'e equation\footnote{For fixed values of parameters.} are done by P. Painlev\'e \cite{painleve1902irreductibilite} and J. Drach \cite{drach1915irreductibilite}, based on the wrong definition of Drach. Computations are now done for all Painlev\'e VI equations \cite{cantat2009dynamics} and for general values of parameters for Painlevé I to V \cite{casaledavy2020}. These works consider fixed values of parameters and compute the Malgrange-Galois groupoid of the corresponding vector field on $\mathbb C^3$. Specifically for the parameters not in the Picard set we have:

\begin{theo}[Cantat-Loray, \cite{cantat2009dynamics}]
For parameters $( a,b,c,e) $ not in Picard parameters set, Malgrange-Galois groupoid is given by:
\begin{align*}
\mathrm{Mal}(\vec X|_{\mathbb C^3 \times \{ ( a,b,c,e)\} } )
&=
\Big\{
\phi : ( \mathbb C^3, *) \overset{\sim}{\longrightarrow} ( \mathbb C^3,\star) \mid \phi^* dx =dx;\,\, \phi_*(\vec X)= \vec X;
\\
& \quad \quad \quad \phi^* dp\wedge dq = dp \wedge dq \mod dx
\Big\}
\end{align*}

In this formula $a,b,c,e$ are fixed parameters. The asterisk and the star stand for arbitrary points in $\mathbb C^3$.
\end{theo}

 These groupoids are ``transversally simple''. Therefore the Malgrange-Galois groupoid of the family has to be a $\mathcal D$-groupoid in $\mathbb C^7$ whose restriction to specific values of parameters yields the ``transversally simple groupoid'' of volume preserving transformations. This situation is similar to the linear case treated in \cite{sanz2018differential}, where the computations were carried out by mean of the so-called Cassidy theorem \cite{cassidy1989classification} on the structure of differential subgroups of simple groups. In the nonlinear situation, one has to replace Cassidy theorem by an algebraic version of  Kiso-Morimoto theorem \cite{kiso1979local}, and apply the non-existence of isomonodromic transformations to get our main result.

\subsection{Organization of the paper} In section \S\ref{s:Dgroupoids} we recall the notion of algebraic pseudogroup (which from now on will be termed $\mathcal D$-groupoid). Our exposition is a simplified version of that in \cite{malgrange2001, malgrange2010pseudogroupes}. This simplification is justified in Appendix \ref{App:A} (see also \cite{Damien2016specialisation}). In section \S\ref{s:Galois} we introduce the Malgrange-Galois groupoid of a vector field, which is a particular case of the notion of Galois groupoid of a foliation treated in \cite{malgrange2001, malgrange2010pseudogroupes}. In \S\ref{sec:PainleveVI} we put together all the pieces to carry on the computation of the Malgrange-Galois groupoid of Painlev\'e VI family. A key point in the computation is an algebraic version of Kiso-Morimoto theorem whose proof is given in Appendix \ref{s:kiso}.

\subsection{Notation and conventions} Given a complex irreducible smooth algebraic variety $M$ we denote $\mathbb C(M)$ the field of rational functions on $M$, $\mathfrak X_M$ the Lie algebra of rational vector fields on $M$ and $\Omega_M^k$ the space of rational $k$-forms on $M$. Given two smooth complex algebraic varieties $M$ and $N$ we denote by $J_k(M,N)\to M\times N$ the bundle of $k$-jets of germs of holomorphic maps from open subsets of $M$ to $N$. ${\rm Aut}_k(M)\subset J_k(M,M)$ is the Zariski open subset of $k$-jets of local biholomorphisms. The varieties $M$ and $N$ can be pointed, so that $J_k((M,p),(N,q))$ denotes the space of $k$-jets of local holomorphisms with source $p$ and target $q$. For a bundle $E\to M$ we use the notation $J_k(E/M)$ and $J(E/M)$ for the bundle $k$-jets and infinite order jets of sections respectively, $J_k(E/M)$ is a subset of $J_k(M,E)$. In any of the above cases the absence of the subindex $k$ means that we take the projective limit of the jet bundles so that $J(M,N)\to M\times N$ represents the bundle of jets of infinite order (or formal maps) from $M$ to $N$, and so on.

\section{$\mathcal D$-groupoids}\label{s:Dgroupoids}

In this section we develop the basic tools for dealing with $\mathcal D$-groupoids. For more details and complete proofs see \cite{malgrange2001, malgrange2010pseudogroupes, casaledavy2020}. Note that the first analytic definition of $\mathcal D$-groupoid in \cite{malgrange2001} was replaced by an algebraic one in \cite{malgrange2010pseudogroupes}. Here we present an alternative definition (Definition \ref{def:D_groupoid}) with the advantage of not referring to the differential ring sheaf structure of the regular functions on ${\rm Aut}(M)$; the equivalence with Definition 5.2 in \cite{malgrange2010pseudogroupes} is given in appendix. Along this section, $M$ is an affine smooth irreducible complex variety 
and $\mathbb C(M)$ is its field of rational functions.

\subsection{Rational groupoids of Gauge transformations}

For the general definitions of grou\-poid, topological groupoid, differentiable groupoid, and Lie groupoid we refer the reader to \cite{mackenzie1987lie, mackenzie2005general}. Let $\pi\colon P\to M$ be a principal bundle modeled over an affine algebraic group $G$. Without loss of generality we assume that $P$ is also an affine and irreducible smooth variety. We consider the groupoid $(s,t)\colon {\rm Iso}(P)\to M\times M$ of $G$-equivariant transformations between the fibers of $\pi$. 
$${\rm Iso}(P) = \left\{\sigma \colon P_x\to P_y \,\colon x,y\in M,\,\,\sigma\,\mbox{ is }\,G\mbox{-equivariant}\right\}.$$
It is an algebraic variety and a Lie groupoid locally isomorphic -in the analytic topology- to $M \times M \times  G$.

\begin{definition}
A rational groupoid of gauge transformations of $P$ is a subvariety $\mathcal G\subseteq {\rm Iso}(P)$ satisfying:
\begin{itemize}
    \item[(a)] There is an open subset $U\subset M$ such that $(s,t)\colon \mathcal G|_U\to U\times U$ is a differentiable groupoid.
    \item[(b)] For any open subset $U\subset M$ $\mathcal G|_U$ is Zariski dense in $\mathcal G$. 
\end{itemize}
\end{definition}

\begin{remark}
Condition (b) means that $\mathcal G$ is defined on the generic point of $M$. It means that irreducible components of $\mathcal G$ dominate $M$ by the source and by the target projections. If $\mathcal G$ is irreducible and satisfies condition (a) then it satisfies condition (b).
\end{remark}

We give three examples showing different pathologies.
\begin{example}
Let us consider $G = \{e\}$, $S$ a  non-trivial closed subset of $M$ with at least two different points, and $U = M\setminus S$. In such case we have ${\rm Iso}(P) = M\times M$. Let us define:
$$\mathcal G = \{(x,x) : x\in M\}\cup (S\times S)$$
It is a subgroupoid of $M\times M$. Moreover $\mathcal G|_{U}$ is a differentiable groupoid, but $\mathcal G|_U$ is not Zariski dense in $\mathcal G$, and therefore $\mathcal G$ is not a rational groupoid.
\end{example}
\begin{example}
Let $M = \mathbb C^n$ $P= M \times \{e\}$ and $P,Q$ two polynomials in $n$ variables such that ${\rm g.c.d.}(P, Q) =1$.
The subvariety  
$$
G =\{(x,y) \in M\times M : P(x)Q(y) = Q(x)P(y)\}
$$
is not a subgroupoid but it is a rational subgroupoid as its restriction on the complement of the indeterminacy locus $ \{P=0, Q=0\}$ and the singular locus $\{QdP-PdQ=0\}$ is a differentiable groupoid.
\end{example}

\begin{example}
Let $M = \mathbb C$ and $P = M\times \mathbb C^\ast$. The subvariety of ${\rm Iso}(P) = M\times M \times \mathbb C^\ast$ given by 
$$ 
G = \{(x,y,u) \in {\rm Iso}(P) : yu=x \}
$$
is a subgroupoid but not a differentiable subgroupoid. It is a differentiable subgroupoid in restriction to $U=\mathbb C^\ast\subset M$.
\end{example}

A useful way of defining a groupoid is through its invariants. Let $F\colon P\to V$ be a rational map. We define the groupoid ${\rm Sym}(F)$ of symmetries of $F$ as the Zariski closure of the set:
$$
\left\{\sigma\in {\rm Iso}(P)\colon 
\forall p\in P_{s(\sigma)}\cap {\rm dom}(F)\, , \,  F(p) = F(\sigma(p))
\right\}
$$
Moreover, for a given subset $\mathbb F\subseteq \mathbb C(P)$ of rational functions we have,
$${\rm Sym}(\mathbb F) = \bigcap_{F\in \mathbb F} {\rm Sym}(F).$$

Reciprocally, given a rational groupoid of Gauge transformations $\mathcal G$ it has an associated field of invariants:
$$
{\rm Inv}(\mathcal G) = \{f\in \mathbb C(P) \colon \forall \sigma\in\mathcal G,\, \forall p\in P_{s(\sigma)}\cap {\rm dom}(f),\, f(p) = f(\sigma(p))\}.
$$
Following \cite[Proposition 2.18]{sanz2018differential},  there is a natural correspondence:

\begin{proposition}[Galois correspondence]\label{pro:Galois_correspondence1}
The assignation $\mathcal G\leadsto {\rm Inv}(\mathcal G)$
is a bijective correspondence (and anti-isomorphism of lattices) between the set of rational groupoids of gauge transformations of $P$ and the set of $G$-invariant subfields of $\mathbb C(P)$ containing $\mathbb C$. Its inverse is given by $\mathbb F \leadsto {\rm Sym}(\mathbb F)$.
\end{proposition}

Let $\vec X$ be a $G$-invariant rational vector field on $P$. Its field of rational first integrals is,
$$\mathbb C(P)^{\vec X} = \{f\in \mathbb C(P) \,\colon\, \vec X \cdot f =0\}.$$
The $G$-invariance of $\vec X$ implies that of the field $\mathbb C(P)^{\vec X}$.

\begin{definition}\label{df:Galois}
The Galois groupoid of the $G$-invariant rational vector field $\vec X$ is,
$${\rm Gal}(\vec X) = {\rm Sym}(\mathbb C(P)^{\vec X}).$$
\end{definition}

\begin{example} 
Some examples of Galois groupoids. 
\begin{enumerate}
\item Let us consider $M = \mathbb C^2$ with coordinates $x$ and $\alpha$, $G = \mathbb C^*$ with coordinate $u$ and $P = M\times G$. Let us consider $\vec X = \frac{\partial}{\partial x} + \alpha u\frac{\partial}{\partial u}$. Then $\alpha$ is a rational first integral of $\vec X$, moreover the integral curves of $\vec X$ with $\alpha \neq 0$ are Zariski dense in planes with constant $\alpha$. Therefore $\mathbb C(P)^{\vec X} = \mathbb C(\alpha)$ and:
$${\rm Gal}(\vec X) = \{(x_1,\alpha_1,x_2,\alpha_2,\sigma) \colon  \alpha_1 = \alpha_2 \} \subset M \times M\times G = {\rm Iso}(P).$$
\item Let us consider $M = \mathbb C$ with coordinate $x$, $G = {\rm GL}_2(\mathbb C)$ with coordinates $u_{11}$, $u_{12}$, $u_{21}$, $u_{22}$ and $P = M\times G$. Let us consider the vector field associated to the fundamental form of Airy equation,
$$\vec X = \frac{\partial}{\partial x} + u_{21}\frac{\partial}{\partial u_{11}} + 
u_{22}\frac{\partial}{\partial u_{12}} + xu_{11}\frac{\partial}{\partial u_{21}} + 
xu_{12}\frac{\partial}{\partial u_{22}}.$$
It is well known that the only rational invariant of this vector field is the determinant function,
$$\Delta = u_{11}u_{22} - u_{12}u_{21}$$
and therefore,
$${\rm Gal}(\vec X) = {\rm Sym}(\mathbb C(\Delta)) = 
\{(x_1,x_2,\sigma)\,\colon\, \sigma\in{\rm GL}_2(\mathbb C),\, {\det}(\sigma) = 1\}.$$
\end{enumerate}
\end{example}

\begin{remark}
Definition \ref{df:Galois} is a particular case of the more general notion of Galois groupoid for a  partial $G$-invariant connection given in \cite[Definition 2.21]{blazquez2019differential}. It suffices to consider that $\langle \vec X \rangle$ as a partial $G$-invariant $\langle \pi_*(\vec X ) \rangle$-connection. The reader may check \cite[\S.2]{blazquez2019differential} for a more detailed account of the definition of the Galois groupoid and further examples.
\end{remark}

\subsection{Frame bundles} 
Let $m$ be the complex dimension of $M$. Let $\Gamma_k$ be the group of $k$-jets of germs of biholomorphisms of $(\mathbb C^m,0)$. They form a projective systems of affine algebraic groups,
$$
\ldots \to \Gamma_k \to \Gamma_{k-1} \to \ldots \to \Gamma_1 \simeq {\rm GL}_m(\mathbb C) \to \{1\}.
$$
The projective limit $\Gamma = \lim_k \Gamma_k$ exists as a pro-algebraic group: the group of formal automorphims of $(\mathbb C^m,0)$.

\begin{definition}
A $k$-frame in $M$ is the $k$-jet at $0\in \mathbb C^m$ of a germ of biholomorphism from $(\mathbb C^m, 0)$ to $ M$. 
\end{definition}

The set of $k$-frames ${\rm R}_kM$ is a variety endowed of a natural projection onto $M$, $j^k_0 \varphi \mapsto \varphi(0)$, and an action of $\Gamma_k$ by composition on the right side that give to ${\rm R}_kM\to M$ the structure of a principal $\Gamma_k$-bundle. The structures for different orders $k$ are compatible, in the sense that the following diagram:
$$
\xymatrix{\Gamma_{k+1} \times {\rm R}_{k+1}M \ar[r]\ar[d] & {\rm R}_{k+1}M\ar[d] \\
\Gamma_{k} \times {\rm R}_{k}M \ar[r] & {\rm R}_{k}M} \quad 
\xymatrix{ (j^{k+1}_0\psi , j^{k+1}_0\varphi)  \ar[r]\ar[d] & j_0^{k+1}(\psi\circ\varphi)\ar[d] \\
(j^{k}_0\psi , j^{k}_0\varphi) \ar[r] & j_0^k(\psi\circ\varphi)}
$$
is commutative. We have a chain of projections,
$$
\ldots \to {\rm R}_kM \to {\rm R}_{k-1}M \to \ldots \to {\rm R}_1M \simeq {\rm L}({\rm T}M) \to M.
$$
The projective limit ${\rm R}M = \lim_k {\rm R}_kM$ exists as a pro-algebraic bundle and it is principal $\Gamma$-bundle. We have a chain of field extensions:
$$
\mathbb C \subset \mathbb C(M) \subset \mathbb C({\rm R}_1M) \subset \mathbb C({\rm R}_1M)  \subset  \ldots \subset \mathbb C({\rm R}M) = 
\bigcup_{k} \mathbb C({\rm R}_kM)
$$
where elements in $\mathbb C({\rm R}_kM)$ are termed \emph{differential functions} of order $k$. The field $\mathbb C({\rm R}M)$ has an additional structure of a differential field with set of derivations $\Delta = \{\delta_1,\ldots,\delta_m\}$, the set of total derivative operators with respect to the cartesian coordinates $\varepsilon_1,\ldots,\varepsilon_m$ in $\mathbb C^m$. This structure will be called $\Delta$-field. To describe these derivations let us introduce the following functions built from $x_1,\ldots,x_m$ is a transcendence basis of $\mathbb C(M)$ : 
$$ 
\begin{array}{rrcl}
\delta^\alpha x_j :&  {\rm R}M &\rightarrow &\mathbb C  \\
  & \varphi  & \mapsto & \partial^\alpha(x_i(\varphi))(0)
\end{array}
$$
Then we will have that the set
$\{\delta^{\alpha}x_j\}_{1\leq |\alpha|\leq k}$ is a transcendence basis and a system of generators of $\mathbb C({\rm R}_kM)$ over $\mathbb C(M)$, thus 
$$
\mathbb C(R_kM) = \mathbb C(M)(\{\delta^{\alpha}x_j\,\,\colon\,\,j = 1,\ldots,m\,\, |\alpha|\leq k\}).
$$
The total derivative operators find their formal expression in coordinates as:
$$
\delta_j = \sum_{|\alpha|\geq 0} \sum_{\ell=1}^m (\delta_j\delta^{\alpha}x_\ell)\frac{\partial}{\partial (\delta^{\alpha}x_\ell)}.
$$

There is a geometric mechanism of extension of vector fields (see, for instance \cite[\S 1.8 p. 29]{malgrange2010pseudogroupes} variational equations of higher order) such that any vector field $\vec X$ in $M$ extends to a unique $\Gamma_k$-invariant rational vector field $\vec X^{(k)}$ in ${\rm R}_kM$ with the properties:
\begin{itemize}
    \item[(1)] $\vec X^{(k)}$ projects onto $\vec X^{(k-1)}$, and $\vec X^{(0)} = \vec X.$
    \item[(2)] The extension commutes with total derivative operators in the following sense: 
    $$\vec X^{(k+1)} \circ \delta_j = \delta_j \circ \vec X^{(k)}.$$
\end{itemize}
Note that if $\vec X$ is a rational vector field in $M$ then $\vec X^{(k)}$ is a rational vector field in ${\rm R}_kM$.

\begin{example}
Let us consider $M = \mathbb C^m$ with coordinate functions $x_i$. In such case ${\rm R}_k\mathbb C^n$ is an affine space coordinate by functions 
$\delta^{\alpha} x_i$ denoted by $x_{i\colon \alpha}$ for short, with 
$\alpha\in \mathbb Z_{\geq 0}^m$ and $|\alpha|\leq k$.
The function $x_{i\colon \alpha}$ represents the derivative of the coordinate functions of $\mathbb C^n$ with respect to the parameters $\varepsilon$ in $(\mathbb C^m,0)$ evaluated in $0$; so that $x_{j:0} = x_j$. The tuple
$(x_{j:\alpha}$ corresponds to the $k$-jet at $0$ of the polynomial map:
$$(\mathbb C^m,0) \to \mathbb C^m, \quad (\varepsilon_1,\ldots ,\varepsilon_m) \mapsto\left(\sum_{|\alpha|\leq k} \frac{x_{1:\alpha}}{\alpha!} \varepsilon_1^{\alpha_1}\cdots \varepsilon_m^{\alpha_m}, \ldots, 
\sum_{|\alpha|\leq k} \frac{x_{m:\alpha}}{\alpha!} \varepsilon_1^{\alpha_1}\cdots \varepsilon_m^{\alpha_m} \right).$$
Total derivative operators are given by formal series:
$$\delta_i = \sum_{\alpha}\sum_{j=1}^m x_{j:\alpha+\epsilon_i}\frac{\partial}{\partial x_{j:\alpha}}.$$
The $k$-th extension of a vector field $\vec X = \sum f_i\frac{\partial}{\partial x_i}$ is given in coordinates as
$$\vec X^{(k)} = \sum_{i=1}^m \sum_{|\alpha|\leq k} (\delta^\alpha f_i)\frac{\partial}{\partial x_{i:\alpha}}.$$
\end{example}

In the limit we have an extension $\vec X^{(\infty)}$ which can be seen as a derivation  of the field $\mathbb C({\rm R}M)$.
The assignation $\vec X\leadsto \vec X^{(\infty)}$ is compatible with the Lie bracket: $[\vec X_1^{(\infty)},\vec X_2^{(\infty)}] =[\vec X_1,\vec X_2]^{(\infty)}$. Let us consider $\mathfrak X_M$ the Lie algebra of rational vector fields on $M$. The field $\mathbb C({\rm R}M)$ has two natural commuting structures of differential field. On one hand, we consider the set $\Delta = \{\delta_1,\ldots,\delta_m\}$ of total derivative operators that endows it with an structure of $\Delta$-fields. On the other hand the Lie algebra $\mathfrak X_M$ acts by derivations in $\mathbb C({\rm R}M)$ endowing it with an structure of $\mathfrak X_M$-field. 

\subsection{Groupoid of automorphisms of $M$} 

\begin{definition}
\label{df:Aut_kM}
One denotes by  $(s,t)\colon \mathrm{Aut}_k(M)\to M\times M$ be the groupoid of $k$-jets of local invertible biholomorphisms of $M$. 
\end{definition}

$\mathrm{Aut}_k(M)$ acts by composition on the $k$-th frame bundle:
$$
\mathrm{Aut}_k(M) \times_M {\rm R_k}M \to {\rm R_k}M \quad (j_p^k \sigma, j_0^k\varphi)\mapsto j_0^k(\sigma \circ \varphi).
$$
A direct calculus show that the groupoid $\mathrm{Aut}_k(M)$ is, in fact, the groupoid of gauge transformations of the $k$-th frame bundle, $\mathrm{Aut}_k(M)\cong {\rm Iso}({\rm R}_k M)$. Given $j_p^k\sigma\in {\rm Aut}_k(M)$ with source $p$ and target $q$. Let us denote $(j_p^k\sigma)^{(k)}$ the induced $\Gamma_k$-equivariant map from $({\rm R}_kM)_p$ and $({\rm R}_kM)_q$:
$$
\xymatrix{
(\mathbb C^m, 0) \ar[r]^-{\varphi} \ar[rd]_-{\psi}&  (M,p) \ar[d]^-{\sigma} \\ & (M,q)
}\quad (j^k_p\sigma)^{(k)}(j_p^k \varphi) = j_0^k\psi.$$
We also have a projective system:
$$
\mathrm{Aut}_k(M) \to \ldots \to \mathrm{Aut}_k(M) \to {\rm Aut}_0(M) = M\times M
$$
where the projective limit, 
$$
{\rm Aut}(M) = \lim_{\leftarrow} {\rm Aut}_k(M)
$$ 
is the pro-algebraic groupoid of formal non-singular maps. As before, $\mathrm{Aut}(M)\cong {\rm Iso}({\rm R}M)$. Given $\sigma\in {\rm Aut}(M)$ with source $p$ and target $q$ we write $\sigma^{(\infty)}$ for the corresponding $\Gamma$-equivariant map from $({\rm R}M)_p$ to $({\rm R}M)_q$. This map $\sigma^{(\infty)}$ is the limit of a sequence of compatible maps;
$\sigma$ is a sequence $\{j^k_p\sigma_k\}$ with $j_k\sigma_k\in {\rm Aut}_k(M)$ such that $j_p^{\ell}\sigma_k = j_p^\ell\sigma_\ell$ for $\ell\leq k$, in the same way a frame $\varphi$ in $p$ is a sequence $\{j^k_0\varphi_k\}$ with $j_0^{\ell}\varphi_k = j_0^\ell\varphi_\ell$ for $\ell\leq k$. Finally $\sigma^{(\infty)}$ is defined by as the limit of the sequence $\{\sigma_k^{(k)}\}$, that is, 
$\sigma^{(\infty)}(\varphi) = \{j^k_0 (\sigma_k\circ \varphi_k) \}$. In the convergent case, when $\sigma = j_p\bar \sigma$ and $\varphi = j_0\bar \varphi$ where $\bar\sigma$ and $\bar\varphi$ are local biholomorphisms we have simply $\sigma^{(\infty)}(\varphi) = j_0(\bar\sigma\circ\bar\varphi)$.

\subsection{$\mathcal D$-groupoids}

By a closed subset in ${\rm Aut}(M)$ we mean a subset that is closed in the initial topology with respect to all projections ${\rm Aut}(M)\to {\rm Aut}_k(M)$. Such closed subset $Z$ is a sequence of Zariski closed subsets $Z = \{Z_k\}_{k\in\mathbb N}$ such that:
\begin{itemize}
    \item[(a)] $Z_k\subseteq {\rm Aut}_k(M)$
    \item[(b)] $Z_{k+1}$ dominates $Z_k$ by projection.
\end{itemize}
With such definition, a formal jet $j_p\varphi$ is in $Z$ if and only if $j_p^k\varphi\in Z_k$ for all $k$. 

\begin{definition}\label{def:D_groupoid}
We say that $\mathcal G$ ($\simeq \{\mathcal G_k\}_{k\in\mathbb N})$ is a $\mathcal D$-groupoid of transformations of $M$ if:
\begin{itemize}
    \item[(a)] For all $k$, $\mathcal G_k$ is a rational subgroupoid of ${\rm Aut}_kM$.
    \item[(b)] For any $f\in {\rm Inv}(\mathcal G_k)$ and $j = 1,\ldots, m$ we have $\delta_jf\in {\rm Inv}(\mathcal G_{k+1})$.
\end{itemize}
\end{definition}

\begin{remark}
$\mathcal D$-groupoids should be seen as spaces of solutions of certain PDE systems. We say that a local biholomorphisms $\sigma$ between open subset of $M$ in $\mathcal G$ if for all $p$ in its domain of definition $j_p\sigma\in\mathcal G$.
\end{remark}

\begin{example}
Let us consider 
$\omega_m = f(x_1,\ldots,x_n)dx_1\wedge\ldots\wedge dx_m$ a rational $m$-form in $\mathbb C^m$. We may consider the subvariety of ${\rm Aut}_1(\mathbb C(M))$ of $1$-jets of biholomorphisms preserving $\omega_m$,
$$\mathcal G_1 = \left\{j_p^1\sigma\in {\rm Aut}_1(\mathbb C^m)|
\sigma^*(\omega_m)(p) = \omega_m(p)\right\}.$$
Its easy to check that $\mathcal G_1$ is a rational subgroupoid of ${\rm Aut}_1(\mathbb C^m)$. If we introduce coordinates,
$$x_i,y_j,y_{i;j} = \frac{\partial y_i}{\partial x_i}$$
in ${\rm Aut}_1(\mathbb C^m)$ then we can compute the equation of $\mathcal G_1$ obtaining:
\begin{equation}\label{eq:exvol}
\mathcal G_1 = \{(x_i,x_j,y_{i;\epsilon_j})\in {\rm Aut}_1(\mathbb C^m)\,\mid\, f(y_1,\ldots,y_m)\det(y_{i;j}) = f(x_1,\ldots,x_m) \}.
\end{equation}
Note that the equation of $\mathcal G_1$ is a first order PDE where $x_1,\ldots,x_m$ are the independent variables and $y_1,\ldots,y_m$ are the unknown functions. In order to derive higher order equations from \eqref{eq:exvol} there are two different ways. 
\begin{enumerate}
    \item[(a)] The classical way is to differentiate equation \eqref{eq:exvol} with respect to the independent variables (see Appendix \ref{App:A}) to obtain the equations of $\mathcal G_2$, $\mathcal G_3$ and so on. 
    \item[(b)] The other way is to differentiate the rational invariants determining $\mathcal G_1$ with respect to the infinitesimal parameters $\varepsilon_i$ to obtain the invariants of $\mathcal G_2$, $\mathcal G_3$ and so on. In order to compute the invariants of $\mathcal G_1$ in $\mathbb C({\rm R}\mathbb C^m)$ we need to interpret the $m$-form $\omega_m$ as a rational function in ${\rm R_1}\mathbb C^m$. For each $1$-frame $j_0^1\varphi$ in $p\in M$ with $p$ in the domain of $f$ we have,
    $$\omega_m(p) \circ d_0\varphi = \lambda(j_0^1\varphi) d\varepsilon_1\wedge\ldots\wedge d\varepsilon_m.$$
    This gives us the rational function $\lambda$ we are looking for $\lambda = \det(x_{i\colon \epsilon_j})f(x_1,\ldots,x_m)$ and $\mathcal G_1 = {\rm Sym}(\mathbb C(\lambda))$. 
    Therefore we obtain $\mathcal G_k = {\rm Sym}(\mathbb C(\delta^\alpha\lambda)_{0\leq|\alpha|\leq k}).$ 
\end{enumerate}
In order to give a complete geometric description of the whole $\mathcal D$-groupoid $\mathcal G$ let us recall than an element $\sigma = [j^1_p\sigma_1, j^2_p\sigma_2, j^3_p\sigma_3, \ldots]\in {\rm Aut}(M)$ with source $p$ and target $q$ can be seen as a formal isomorphism from $(\mathbb C^m,p)$ to $(\mathbb C^m,q)$. We see $\sigma^*(\omega_m)$ as a formal $m$-form near $p$:
$$\sigma^*(\omega_m) = [\sigma_1^*(\omega_m)(p), j^1_p\sigma_2^*(\omega_m)(p), j^2_p\sigma_3^*(\omega_m)(p), \ldots ]$$
and then we may give the description of $\mathcal G$ as the $\mathcal D$-groupoid of formal symmetries of $\omega_m$,
that is, its elements are the formal isomorphisms $\sigma$ such that $\sigma^*(\omega_m) = j_p\omega_m$.
Moreover, as a rational $m$-form is completely determined by its infinite order jet at any point of its domain it is usual to identify $\omega_m$ with its jet $j_p\omega_m$ and just write: 
$$\mathcal G = \left\{\sigma\in {\rm Aut}(\mathbb C^m)\,|\,
\sigma^*(\omega_m) = \omega_m\right\}.$$
\end{example}

Given a $\mathcal D$-groupoid $\mathcal G$ we define its field of differential invariants as:
$$
{\rm Inv}_\Delta(\mathcal G) = \bigcup_k {\rm Inv}(\mathcal G_k) \, \subseteq \,\, \mathbb C({\rm R}M).
$$
Condition (b) implies that ${\rm Inv}_\Delta(\mathcal G)$ is a $\Delta$-subfield of $\mathbb C({\rm R}M)$. Reciprocally, given a $\Delta$-subfield $\mathbb F$ we can define its $\mathcal D$-groupoid of symmetries ${\rm Sym}_\Delta(\mathbb F) \simeq \{{\rm Sym}_\Delta(\mathbb F)_k\}_{k\in\mathbb N}$ with:
$$
{\rm Sym}_\Delta(\mathbb F)_k = {\rm Sym}(\mathbb F \cap \mathbb C({\rm R}_kM) ).
$$
We have a Galois correspondence between $\mathcal D$-groupoids and their $\Delta$-fields of differential invariants.

\begin{proposition}[$\Delta$-Galois correspondence]\label{pro:Galois_correspondence2}
The assignation $\mathcal G\leadsto {\rm Inv}_\Delta(\mathcal G)$ is a bijective correspondence (and anti-isomorphism of lattices) between the set of $\mathcal D$-groupoids of transformations of $M$ and  $\Gamma$-invariant $\Delta$-subfields of $\mathbb C(P)$ containing $\mathbb C$. Its inverse is given by $\mathbb F\leadsto {\rm Sym}_\Delta(\mathbb F)$
\end{proposition}

\begin{proof}
Note that if $\mathbb F$ is a $\Gamma$-invariant $\Delta$-subfield of $\mathbb C({\rm R}M)$ containing $\mathbb C$ then the intersection $\mathbb F\cap \mathbb C({\rm R}_kM)$ is a $\Gamma_k$-invariant subfield of $\mathbb C({\rm R}_kM)$ containing $\mathbb C$. The proposition follows by application of Proposition \ref{pro:Galois_correspondence1}.
\end{proof}

Le $\mathcal G =\{\mathcal G\}_k$ be a $\mathcal D$-groupoid in $M$. By a finiteness theorem of Kolchin \cite[Proposition 14, p 112]{kolchin1973} the $\Delta$-field ${\rm Inv}_{\Delta}(\mathcal G)$ is $\Delta$-finitely generated. There is a minimum $r$ such that ${\rm Inv}_{\Delta}(\mathcal G)$ in $\Delta$-generated by ${\rm Inv}(\mathcal G_r)$. This minimum $r$ is the order of the $\mathcal D$-groupoid $\mathcal G$. Moreover, let $U\subset M$ be a Zariski open subset such that $\mathcal G_r|_U$ is a groupoid,  we have that $\mathcal G|_U$ is a groupoid.

\subsection{$\mathcal D$-Lie algebras} 

Let us now consider the jet bundle $J({\rm T}M/M)\to M$. Its elements are formal vector fields in $M$ 
(that is, continuous derivations of the completed local rings $\hat{\mathcal O}_{M,p}$ for $p\in M$)
and thus it is a bundle by Lie algebras with the usual Lie bracket of formal vector fields. Therefore, the space of rational sections
$$
\Gamma_{\rm rat}(J({\rm T}M/M)) = \lim_{\leftarrow} \Gamma_{\rm rat}(J_k({\rm T}M/M))
$$
is a $\mathbb C(M)$-Lie algebra.

\begin{remark}
There is also a Lie algebroid structure for each order $J_k({\rm T}M/M)$, and then a natural Lie algebra structure in each space of sections $\Gamma_{\rm rat}(J_k({\rm T}M/M))$ given by the so-called Spencer bracket. This Lie algebra structure is different from the one we consider, but they coincide along holonomic sections. 
\end{remark}

\begin{definition}
A rational linear sub-bundle of $J_k({\rm T}M/M)$ is an irreducible Zariski closed subset $V\subset J_k({\rm T}M/M)$ such that there is a Zariski open subset $U\subseteq M$ such that $V|_U\to U$ is a vector bundle. 
\end{definition}

\begin{example}
Order $0$ rational linear sub-bundles of $J_0(TM/M) = TM$ are singular distributions of vector fields $\mathscr F\subset TM$. If $U$ is the complement of the singular set of $\mathscr F$ then $\mathscr F|_U\to U$ is a vector bundle. 
\end{example}

As linear bundles are rationally trivial, a rational linear sub-bundle is characterized by its space of rational sections $\Gamma_{\rm rat}(V_k)\subset \Gamma_{\rm rat}(J_k({\rm T}M/M))$. There is a natural bijective correspondence between rational linear sub-bundles of $J_k({\rm T}M/M)$ and $\mathbb C(M)$-subspaces of $\Gamma_{\rm rat}(J_k({\rm T}M/M)).$ In the projective limit $J({\rm T}M/M)$ we consider the initial topology. Thus, a closed subset $V\subset J({\rm T}M/M)$ is a sequence $V = \{V_k\}_{k\in \mathbb N}$ with $V_k$ Zariski closed in $J_k({\rm T}M/M)$ and such that 
$V_{k+1}$ dominates $V_{k}$ by projection. 

\begin{definition}
A rational linear sub-bundle of $J({\rm T}M/M)$ is an irreducible closed subset $V = \{V_k\}_{k\in \mathbb N}$ such that for all $k$ we have that $V_k$ is a rational linear sub-bundle of $J_k({\rm T}M/M)$.
\end{definition}

Analogously, there is a natural bijective correspondence between rational linear sub-bundles of $J({\rm T}M/M)$ and $\mathbb C(M)$-subspaces of $\Gamma_{\rm rat}(J({\rm T}M/M))$. In what follows we will define the notion of $\mathcal D$-Lie algebra, a certain kind of $\mathbb C(M)$-subspace of \linebreak $\Gamma_{\rm rat}(J({\rm T}M/M))$; but we will apply the notion indistinctly to rational linear sub-bundles of $J({\rm T}M/M)$.

\begin{definition}
The ring of rational differential operator in $M$ is the ring $\mathcal D_M$ of $\mathbb C$-linear endomorphisms of $\mathbb C(M)$ generated by:
\begin{itemize}
    \item[(a)] $\mathbb C(M)$ acting by multiplication,  that is, each element $h\in\mathbb C(M)$ is seen as an endomorphism $h\colon \mathbb C(M)\to \mathbb C(M)$, $f\mapsto hf$.
    \item[(b)] The Lie algebra of rational vector fields $\mathfrak X_M$,  that is, rational vector field $\vec X\in\mathfrak X_M$ is seen as a derivation $\vec X\colon \mathbb C(M)\to \mathbb C(M)$, $f\mapsto \vec Xf$. 
\end{itemize}
\end{definition}

We set the degree of rational functions equal to $0$ and that of vector fields to be $1$. Thus, $\mathcal D_M$ is a non commutative graded ring, for $h\in\mathbb C(M)$ and $\vec X\in \mathfrak X_M$ we have,
$$f\circ \vec X = f\vec X, \quad \vec X \circ f = \vec Xf + f\vec X.$$
The inclusion $\mathbb C(M)\subset \mathcal D_M$ endows $\mathcal D_M$ with a $\mathbb C(M)$-bimodule structure with a
left ($f\circ\theta\colon g \mapsto f(\theta(g))$) and a right ($\theta\circ f\colon g \mapsto \theta(fg)$) multiplication.

Let $\Omega_M^1 = \mathfrak X_M^*$ be the $\mathbb C(M)$-space of rational $1$-forms in $M$. We define the space of rational differential operators 
from $\mathfrak X_M$ to $\mathbb C(M)$ as the tensor product:
$$
{\rm Diff}(\mathfrak X_M,\mathbb C(M)) = \mathcal D_M\tens_{\mathbb C(M)} \Omega_M^1.
$$
Note that the tensor product is constructed by means of the right $\mathbb C(M)$-module structure in $\mathcal D_M$. Thus, ${\rm Diff}(\mathfrak X_M,\mathbb C(M))$ is a left $\mathcal D_M$-module. \\

Differential operators can be seen as rational functions on $J({\rm T}M/M)$ that are linear along fibers of the projection $J({\rm T}M/M)\to M$. 
The coupling of a rational differential operator and a formal vector field $\vec X\in J({\rm T}M/M)$ with base point $p$ is given by $(L \otimes \omega)(\vec X) = L(\omega(X))(p)$. Note that, if $p$ is in the domain of $\omega$ and $L$ and $\omega(X)$ and $L(\omega(\vec X))$ are formal functions at $p$. When we couple a rational differential operator with a rational section of $J({\rm T}M/M)$ we obtain a rational function. Moreover we have a duality:
$$
{\rm Diff}(\mathfrak X_M,\mathbb C(M))^* = \Gamma_{\rm rat}(J({\rm T}M/M)).
$$ 

\begin{definition}\label{def:DLie_alg}
Let $\mathcal L = \{\mathcal L_k\}_{k\in \mathbb N}$ be a rational linear sub-bundle of $J({\rm T}M/M)$. We say that $\mathcal L$ is a $\mathcal D$-Lie algebra of transformations of $M$ if:
\begin{itemize}
    \item[(a)] Its space of sections $\Gamma_{\rm rat}(\mathcal L)$ is a Lie subalgebra of $\Gamma_{\rm rat}(J({\rm T}M/M))$.
    \item[(b)] Its $\mathcal D_M$-module of vanishing differential operators:
    $$
    {\rm ann}(\mathcal L) = \{\theta\in {\rm Diff}(\mathfrak X_M,\mathbb C(M))\,\colon \, \forall \vec X \in \mathcal L \,\,\, \theta(\vec X) = 0 \}
    $$ 
    is a $\mathcal D_M$-submodule of ${\rm Diff}(\mathfrak X_M,\mathbb C(M))$.
\end{itemize}
\end{definition}

\begin{remark}
$\mathcal D$-Lie algebras should be seen as spaces of solutions of certain linear PDE systems. Let $\mathcal L = \{\mathcal  L_k\}_{k\in \mathbb N}$ be a $\mathcal D$-Lie algebra in $M$ seen as a linear sub-bundle of $J({\rm T}M/M)$. We say that a local analytic vector field $\vec X$ is in $L$ if for all $p$ in its domain of definition $j_p\vec X\in\mathcal L$ (or equivalently, for all $p$ and $k$ $j_p^k\vec X\in \mathcal L_k$). 
\end{remark}

\begin{remark}
For the sake of simplicity we presented the definition of $\mathcal D$-Lie algebra in terms of $\mathbb C(M)$-vector spaces instead as coherent sheaves of differential operators as it is done in \cite{malgrange2001} \S 3, \cite{malgrange2010pseudogroupes} \S 5.6 or \cite{le2010algebraic} \S 4. However, it is clear that the $\mathbb C(M)$-space of rational differential operators ${\rm ann}(L)$ in Definition \ref{def:DLie_alg} can be seen as a coherent $\mathcal D$-module of differential operators. In particular we have the following results:
\begin{itemize}
    \item[(a)] The $\mathcal D$-Lie algebra and its module of vanishing differential operators ${\rm ann}(\mathcal L)$ determine each other.
    \item[(b)] There is an open subset $U\subset M$ such that for all $k$ $L_k|_U$ is a linear bundle over $U$ (Proposition 4.1 in \cite{le2010algebraic}).
\end{itemize}
\end{remark}

\subsection{$\mathcal D$-Lie algebra of a $\mathcal D$-groupoid}

Given a $\mathcal D$-groupoid of transformations of $M$ its Lie algebra is usually defined by its infinitesimal generators: vector fields $\vec X$ such their exponential ${\rm exp}(t\vec X)$ is in $\mathcal G$ wherever it is defined.


The following definition, using differential invariants, is easier to use in the algebraic situation than the usual definition. The compatibility of the prolongation of $\vec X$ to $\vec X^{\infty}$ with the flow and the $\Delta$-Galois correspondence \ref{pro:Galois_correspondence2} implies the equivalence of the two point of view.

\begin{definition}
\label{def:lie_D_algebra}
Let $ \mathcal G = \{\mathcal G_k\}_{k\in \mathbb N}$  be a $\mathcal D$-groupoid of transformations of $M$. 
The $\mathcal D$-Lie algebra of $\mathcal G$ is 
$$
\mathrm{Lie}\, (\mathcal G) =  \big\{ \vec X\in J({\rm TM}/M) \,\colon \,\vec X^{(\infty)} {\rm Inv}_\Delta(\mathcal G) = 0 \big\}
$$
\end{definition}

It is clear that as a sequence of rational linear sub-bundles we have ${\rm Lie}(\mathcal G) = \{{\rm Lie}(\mathcal G_k)\}_{k\in \mathbb N}$ where: 
$$
{\rm Lie}(\mathcal G_k) = \big\{j_p^k X\in J_k({\rm TM}/M) \,\colon \,X^{(k)} {\rm Inv}(\mathcal G_k) = 0 \big\}
$$
Because $[\vec X_1^{(\infty)}, \vec X_2^{(\infty)}] =[\vec X_1, \vec X_2]^{(\infty)}$, the space of rational sections of ${\rm Lie}(\mathcal G)$ is closed by Lie bracket. Therefore, it is a $\mathcal D$-Lie algebra. 

\begin{example}
Let us consider $\mathcal G$ the $\mathcal D$-lie groupoid of formal maps preserving an $m$-form $\omega_m$:
$$\mathcal G = \{\sigma\in {\rm Aut}(M) \,\mid\, \sigma^*(\omega_m) = \omega_m\}$$
its associated $\mathcal D$-Lie algebra is the $\mathcal D$-Lie algebra of formal infinitesimal symmetries of $\omega_m$:
$${\rm Lie}(\mathcal G) = \{\vec X\in J({\rm TM}/M) \,\mid\, {\rm Lie}_{\vec X}\omega_m = 0\}.$$
where, given a formal vector field $\vec X = [\vec X_0(p), j^1_p\vec X_1, j^1_p\vec X_1, \ldots ]$ the Lie derivative of $\omega_m$ is defined as a formal $m$-form near $p$,
$${\rm Lie}_{\vec X}(\omega_m) = [{\rm Lie}_{\vec X_1}(\omega_m)(p), j^1_p{\rm Lie}_{\vec X_2}(\omega_m)(p),
j^2_p{\rm Lie}_{\vec X_3}(\omega_m)(p), \ldots].$$
Moreover, from previous works \cite[Theorem 1.3.2]{casale2004groupoide}  we know that transitive $\mathcal D$-groupoids are groupoids of formal symmetries of geometric structures. Their associated $\mathcal D$-Lie algebras are the $\mathcal D$-Lie algebras of formal infinitesimal symmetries of such geometric structures.
\end{example}

\section{Malgrange-Galois $\mathcal D$-groupoid}\label{s:Galois}

\subsection{Malgrange-Galois $\mathcal D$-groupoid of a vector field}
In this section $M$ is an affine smooth algebraic variety of dimension $m$. Let $\vec X$ be a rational vector field on $M$. Let us recall that for all $k$ the $k$-th extension $\vec X^{(k)}$ is a $\Gamma_k$-invariant vector field in $\mathbb C({\rm R}_kM)$ and $\vec X^{(\infty)}$ is a $\Gamma$-invariant derivation of $\mathbb C({\rm R}M)$.

\begin{definition}
The field of rational differential invariants of $\vec X$ is the field of constants of the derivation $\vec X^{(\infty)}$:
$$\mathbb C({\rm R}M)^{\vec X} = \{f\in \mathbb C({\rm R}M) \,\colon\, \vec X^{(\infty)}f = 0\}.$$
\end{definition}

Let us list some elementary self evident properties of $\mathbb C({\rm R}M)^{\vec X}$:
\begin{enumerate}
\item[(a)] As the derivation $\vec X^{\infty}$ commutes with the total derivative operators $\Delta = \{\delta_1,\ldots,\delta_m\}$ it follows that $\mathbb C({\rm R}M)^{\vec X}$ is a $\Delta$-subfield.
\item[(b)] Since for all $k$, $\vec X^{\infty}|_{\mathbb C({\rm R}_kM)} = \vec X^{(k)}$ we have that the field of rational differential invariants of order $\leq k$,
$$\mathbb C({\rm R}_kM)^{\vec X} = \mathbb C({\rm R}_kM)\cap \mathbb C({\rm R}M)^{\vec X}$$
is the field of rational first integrals of the vector field $\vec X^{(k)}$.
\item[(c)] Clearly:
$$\mathbb C({\rm R}M)^{\vec X} = \bigcup_{k} \mathbb C({\rm R}_kM)^{\vec X}.$$
\item[(d)] As $\vec X^{(\infty)}$ is $\Gamma$-invariant the $\Delta$-field of rational differential invariants $\mathbb C({\rm R}M)^{\vec X}$ is $\Gamma$-invariant.
\end{enumerate}

\begin{definition}\label{df:Mal_groupoid}
The Malgrange-Galois $\mathcal D$-groupoid of $\vec X$ is: 
$${\rm Mal}(\vec X)= {\rm Sym}_\Delta(\mathbb C({\rm R}M)^{\vec X}),$$  the $\mathcal D$-Lie groupoid corresponding by means of the $\Delta$-Galois correspondence (Proposition \ref{pro:Galois_correspondence2}) to the $\Delta$-field $\mathbb C({\rm R}M)^{\vec X}$. 
\end{definition}

The reader may be aware that this definition looks different from that in \cite{malgrange2001, malgrange2010pseudogroupes}, however they are equivalent, as discussed in appendix A. Indeed, the equivalence between three different definitions of ${\rm Mal}(\vec X)$, including Definition \ref{df:Mal_groupoid}
can be found in \cite[Th\'eor\`eme 3.16]{casaledavy2020}.

Computation of Malgrange-Galois $\mathcal D$-groupoid is in general very difficult. In \cite{casale2011}, the reader may found explicit examples written in coordinates such as the computation for linear differential equations or for first order differential equation. An easy case is given in the example \ref{ex:integrable} below.

\begin{remark}
Applying the definition of ${\rm Sym}_\Delta(\mathbb C({\rm R}M)^{\vec X})$ we may write a direct description of the Malgrange-Galois groupoid:
$${\rm Mal}(\vec X) = \left\{\sigma\in {\rm Aut}(M)\,\colon\,  
\forall f\in\mathbb C({\rm R}M)^{\vec X}\,\,  f \circ \sigma^{(\infty)}= f|_{({\rm} RM)_{s(\sigma)}}
\right\},$$
the Malgrange-Galois groupoid of $\vec X$ is the $\mathcal D$-groupoid of transformations of $M$ that fixes the rational differential invariants of $\vec X$. Remember that $\sigma^{(\infty)}$ is an ismorphism of ${\rm R}M_{s(\sigma)}$ on ${\rm R}M_{t(\sigma)}$.
\end{remark}

\begin{example}{\label{ex:integrable}}
The simplest Malgrange-Galois group is that of a completely integrable vector field. Let us consider $M = \mathbb C^m$
$\vec X = \partial_{x_1}$. In such case we have $\mathbb C(M)^{\vec X} = \mathbb C(x_2,\ldots,x_m)$. Incidentally the extension of $\vec X$ has the same coordinate expression $\vec X^{(\infty)} ={\partial}_{x_1}$ and therefore all the functions $x_{i:\alpha}$ with $|\alpha|>1$ are rational differential invariants of $\vec X$. Elements $\sigma\in {\rm Aut}(\mathbb C^m)$ preserving all the differential functions $x_{i:\alpha}$ is necessarily jets of a translations. Additionally, the invariance of $x_2,\ldots,x_m$ implies that such translation is done along the axis $x_1$. Therefore:
$${\rm Mal}(\vec X) = \{ j_p\tau_\lambda\,\mid p\in \mathbb C^m, \lambda\in \mathbb C\}$$
where $\tau_\lambda \colon (x_1,x_2\ldots,x_n)\to (x_1+\lambda,x_2\ldots,x_n)$ is the translation of magnitude $\lambda$ along the $x_1$ axis. 
\end{example}

If is clear that as a sequence of rational subgroupoids of ${\rm Aut}_k(M)$ we have ${\rm Mal}(\vec X) = \{{\rm Mal}_k(\vec X)\}_{k\in\mathbb N}$ where 
${\rm Mal}_k(\vec X)$ is the rational subgroupoid of ${\rm Aut}_k(M)$ whose field of invariants is $\mathbb C({\rm R}_kM)^{\vec X}$. Therefore:
$${\rm Mal}_k(\vec X) = {\rm Gal}(\vec X^{(k)}).$$
This ${\rm Mal}_k(\vec X)$ is termed the Malgrange-Galois groupoid of order $k$ of $\vec X$.

\subsection{Specialization theorem}
\label{sec:specialization_theo} Let us examine a parametric version of nonlinear Galois theory and a specialization result stated in \cite[\S4]{casaledavy2020}. We explain the relations between the formalism of \cite{casaledavy2020} and ours and state the theorem referring to \cite{casaledavy2020} for the proof.

Let $\rho\colon M\to S$ be a surjective smooth map of smooth affine varieties with smooth irreducible fibers of dimension $r$. For each $s\in S$ we consider $M_s= \rho^{-1}(\{s\})$. Let $\vec X$ be a rational vector field tangent to $\rho$ \emph{i.e.} such that $d\rho(\vec X)=0$. We assume (by replacing $M$ by an affine open subset if necessary) that $\vec X$ restricts to a rational vector field $\vec X|_{M_s}$ at each fiber $M_s$. The bundle of partial frames of $M$ with respect to $\rho$ of order $k$, is defined as
$$
\mathrm{R}_k(M/S) = \bigcup_{s\in S} {\mathrm R}(M_s). 
$$
which turns out to be an affine variety and a principal bundle over $\Gamma'_k = \textrm{Aut}_k(\mathbb C^r ,0)$ with $r = \dim(M/S)$. The groupoid of gauge isomorphisms of ${\rm R}(M/S)$ turn out to be:
$${\rm Aut}_k(M/S) = \bigcup_{s\in S} {\rm Aut}_k(M|_s),$$
which is an algebraic groupoid on $M$ of transformations between fibers of $\rho$.
The rational vector field $\vec X$ can be extended to a $\Gamma_k'$-invariant rational vector field $(\vec X/S)^{(k)}$ in ${\rm R}_k(M/S)$
$$(\vec X/S)^{(k)}(p) = \left (\vec X|_{M_{\rho(p)}}\right)^{(k)}.$$
As in the general theory, we may take the limit in $k$ obtaining the bundle of partial frames of $M$ with respect to $\rho$, 
$$
\mathrm{R}(M/S) = \bigcup_{s\in S} {\mathrm R}(M_s). 
$$
which turns out to be a principal bundle over $\Gamma_k = \textrm{Aut}(\mathbb C^r ,0)$. The field of rational functions
$\mathbb C({\rm R}(M/S))$ is also endowed with the total derivative operators $\Delta' =\{\delta_1,\ldots,\delta_r\}$ and it is a $\Delta'$-field. The derivation ${(\vec X/S)}^{(\infty)}$ defined as,
$$(\vec X/S)^{(\infty)}f = (\vec X/S)^{(k)}f\quad \mbox{for} \quad f\in \mathbb C({\rm R}_k(M/S)),$$
commutes with $\Delta'$. The groupoid of gauge isomorphisms\footnote{It can be seen as a quotient of the algebraic subgroupoid ${\rm Sym}_{\Delta}(\mathbb C\langle \rho^*\mathbb C(S)\rangle_{\Delta}) \subset {\rm Aut}_k(M)$)
} of ${\rm R}(M/S)$ is,
$${\rm Aut}(M/S) = \bigcup_{s\in S} {\rm Aut}(M_s).$$
The partial Malgrange groupoid is defined for each order,
$$
\mathrm{Mal}(\vec X/S)= \lim_{\leftarrow} \mathrm{Gal}((\vec X/S)^{(k)}).
$$
For each $s\in S$ we have that $\textrm{Aut}( M_s)$ is a Zariski closed subset of $\mathrm{Aut}(M/S)$. Thus, we can speak of the restriction of the partial Galois groupoid to a fibre,
$$
\mathrm{Mal}(\vec X/S)|_{M_s} = 
\mathrm{Mal}(\vec X/S) \cap \textrm{Aut}(M_s).
$$
Note that this restriction is $\mathcal D$-groupoid in $M$ for generic values of $s$.
The following is one of the statements in \cite[Th\'eor\`eme 4.8]{casaledavy2020} \footnote{The reader should be aware that the notation in the reference is different. There the Malgrange-Galois groupoid appears as ${\rm Gal}$ and not ${\rm Mal}$.}.

\begin{theorem}
\label{th:specialisation}
For all $s\in S$,
$\mathrm{Mal}(\vec X|_{M_s})\subset \mathrm{Mal}(\vec X/S)|_{M_s}.$
\end{theorem}

\subsection{Projection theorem}

Let $\pi\colon P\to M$ be a $G$-principal bundle. Assume we have a normal algebraic subgroup $K$ of $G$ and write $\overline{G} =G/K$. Let us consider  $q\colon P\to \overline{P} = P/K$, $p\mapsto \overline{p}$. Then $\overline P$ is a $\overline{G}$-principal bundle over $M$. The actions commute: $\overline{pg} = \overline{p}\overline{g}$. There is also a projection $q_*$ of groupoids,
$$
\xymatrix{ 
\mathrm{Iso}\, P \ar[r]^-{q_*} \ar[dr]_{s,t} 
& \mathrm{Iso}\, \overline{P}\ar[d]^{s,t} 
\\
&M
}
$$
where the isomorphism $\sigma_{p,q}$ (that sends $p\in P$ to $q\in P$) is sent to $\sigma_{\overline{p}, \overline{q}}$.
\begin{theorem}
\label{th:projection_GGal_theorem}
Let $\vec X$ be a rational $G$-invariant vector field in $P$. Then
$$q_*({\rm Gal}(\vec X)) = {\rm Gal}(q_*\vec X).$$
\end{theorem}

\begin{proof}
It is a particular statement of a result in \cite{casale2009preuve}. However it suffices to note that $q^*(\mathbb C(\overline P)^{q_*\vec X}) \subseteq \mathbb C(P)^{\vec X}$.
\end{proof}

Now, as in subsection \S\ref{sec:specialization_theo}, let us consider $\rho\colon M\to S$ with $\dim M = m$ and $\dim M/S =r$. Let $G$ be the subgroup of $\Gamma$ of maps that leave $ \mathbb C^r\subset \mathbb C^m$ invariant, where we identify $\mathbb C^r$ inside $\mathbb C^m$ as $ \mathbb C^r = \{ \varepsilon_{r+1} =\cdots = \varepsilon_{m}=0 \}$. Explicitly we have 
$$
G= \Big\{ \phi : (\mathbb C^m,0)\to (\mathbb C^m,0) 
\,\colon \,
\frac{\partial \phi_i}{\partial \varepsilon_j} =0 \mbox{ for } i=1,\ldots, r; \quad j=r+1,\ldots, m\Big\}$$

Let us consider $\mathrm{R}^{\rho}_k M$ the set of $k$-jets of biholomorphisms $(\mathbb C^m,0)\to (M,p)$ that send the subspace $\mathbb C^r$ to the fiber $M_{\rho(p)}$. If a frame is in $\mathrm{R}^{\rho}_k M$ then it can be restricted to $(\mathbb C^r,0)$ obtaining a frame of the fiber. By taking projective limit we obtain  $ \mathrm{R}^\rho M\subset \mathrm{R}M$ as a $G$-sub-bundle, a reduction of structure group of $\mathrm{R}M$ from $\Gamma$ to $G$.
The gauge groupoid $ \mathrm{Iso}(\mathrm{R}^\rho M )$is identified with $\textrm{Aut}(M)^\rho = {\rm Sym}_\Delta(\rho^*\mathbb C(S))$ the $\mathcal D$-groupoid of formal maps respecting the projection $\rho$. As $d \rho(X)=0$ we have $X^{(\infty)}$ is tangent to $ \mathrm{R}^\rho M$ and $ \mathrm{Mal}(\vec X) \subset\textrm{Aut}(M)^\rho$. The group $G$ acts on $  \mathrm{R}^\rho M$.  There is a natural exact sequence,
$$ 
0 \to K \to G\to \Gamma' \to 0
$$
given by the restriction to $\mathbb C^r$. Here, $K$ is the subgroup of formal maps in $\mathbb C^m$ inducing the identity in $\mathbb C^r$. A frame in $\mathrm{R}^{\rho}M$ can be restricted to $\mathbb C^m$, and therefore we obtain a frame on a fiber of $\rho$. Thus, we have a projection,
$$
 \mathrm{R}^\rho M \longrightarrow 
 \mathrm{R}^\rho(M)/K \simeq
 \mathrm{R}(M/S), \quad
 \phi \mapsto \phi|_{ \mathbb C^m}.
$$
It is straightforward that $X^{(\infty)}$ is projectable and projects onto $(X/S)^{\infty}$. 
By Theorem \ref{th:projection_GGal_theorem} (applied to all finite orders $k$) we get a surjective map,
$$
\mathrm{Mal}(\vec X) \longrightarrow \mathrm{Mal}(\vec  X/S).
$$
 Let $\sigma\in {\rm Aut}(M)^{\rho}$ and let $s$ be the projection of the source of $\sigma$. The target is $\sigma$ is also in $M_s$ and moreover $\sigma$ restricts to a formal map in ${\rm Aut}(M_s)$. If we fix a fiber $M_s$, taking into account that ${\rm Mal}(\vec X)\subset{\rm Aut}(M)^{\rho}$ then it makes sense to consider the restriction ${\rm Mal}(\vec X)|_{M_s}$ as the restriction to $M_s$ of all elements of ${\rm Mal}(\vec X)$ with source in $M_s$. From the surjectiveness of the above map we obtain $\mathrm{Mal}(\vec X)|_{M_s} = \mathrm{Mal}(\vec X/S)|_{M_s}$. Now, by application of Theorem \ref{th:specialisation} to the right and side of the equality we get: 

\begin{corollary}
\label{co:restriction_Mal_fiber}
For all $s\in S$, $ \mathrm{Mal}(\vec X)|_{M_s} \supset \mathrm{Mal}(\vec X|_{M_s})$.
\end{corollary}

\section{Malgrange-Galois groupoid of Painlev\'e VI equation}
\label{sec:PainleveVI}

\subsection{Hamiltonian form for Painlev\'e VI}

 Painlev\'e VI equation \index{Painlev\'e VI equation} is a second order differential equation for function $u$ of a complex variable $x$ of the form (see, for example, \cite[ p.119]{iwasaki2013gauss}, and \cite{okamoto1980polynomial})
\begin{align}
\label{eq:Pain_VI}
u'' &=F(x,u,v, a,b,c,e); \quad u'=v
\end{align}
where $F\in \mathbb C(x,u,v,a,b,c,e)$ is
\begin{eqnarray*}
\frac{1}{2} \left( \frac{1}{u} + \frac{1}{u-1} + \frac{1}{u-x} \right){v}^2 -\left( \frac{1}{x} + \frac{1}{x-1} + \frac{1}{u-x} \right){v}
\\
+ \frac{ u(u-1)(u-x)}{x^2(x-1)^2} \left( \frac{c^2}{2} - {\frac{a^2}{2}}\frac{x}{u^2} + {\frac{b^2}{2}}\frac{x-1}{(u-1)^2} + \frac{1-e^2}{2} \frac{x(x-1)}{(u-x)^2} \right).
\end{eqnarray*}
The variables $a$, $b$, $c$, $e$ are the parameters of Painlev\'e VI equation. However, our interest is to consider the role of the parameters in nonlinear differential Galois theory. Therefore, we see Painlev\'e VI equation as the following rational vector field in $\mathbb C^7$,
\begin{align}
\vec Y = \frac{\partial}{\partial x} + v \frac{\partial}{\partial u}  + F(x,u,v, a,b,c,e) \frac{\partial}{\partial v}. \tag{$\rm P_{VI}$}
\end{align}
The trajectories of this vector field, parameterized by $x$ are $(x,u(x), u'(x), a,b,c,e)$ for $u$ a solution of $P_{VI}$  with fixed parameters.
Equation Painlev\'e VI \eqref{eq:Pain_VI} admits the following equivalent Hamiltonian form (see \cite[ p.140]{iwasaki2013gauss})
with hamiltonian function
\begin{align*}
H &= \frac{1}{x(x-1)} \Big[ p(p-1)(p-x)q^2 - \Big( a(p-1)(p-x) + b p(p-x) +(e-1)p(p-1)\Big) q \\
&\quad \quad + \frac{1}{4}\big( (a+b+e-1)^2 -c^2\big) (p-x) \Big].
\end{align*}
and hamiltonian vector field
\begin{align}
\label{eq:PainVI_vec_field2}
\vec X = \frac{\partial}{\partial x} + \frac{\partial H}{\partial q} \frac{\partial}{\partial p}  -\frac{\partial H}{\partial p}\frac{\partial}{\partial q}.
\tag{$\rm HP_{VI}$}
\end{align}

The dominant finite map $\phi: \mathbb C^7 \to \mathbb C^7$
$$
(x,p,q,a,b,c,e) \mapsto \left(x,p, \frac{\partial H}{\partial q}, \frac{c^2}{2}, \frac{a^2}{2}, \frac{b^2}{2}, \frac{e^2}{2} \right)
$$
sends the vector field $\vec X$ onto the vector field $\vec Y$, giving the equivalence between systems \eqref{eq:Pain_VI} and \eqref{eq:PainVI_vec_field2}, thus a conjugation of their Malgrange-Galois groupoids. Our purpose is to compute the Malgrange-Galois groupoid of vector field $\vec X$, \eqref{eq:PainVI_vec_field2}.


\subsection{Some invariants of the Malgrange-Galois groupoid for Painlev\'e VI}
Along this section let us consider the following diagram of trivial bundles:
$$
\xymatrix{ 
M=\mathbb C^7 \ar[r]^-{\pi} \ar[dr]_-{\bar\rho} & \mathbb C^5_{x,a,b,c,e}=B \ar[d]^-{\rho} \\ 
&\mathbb C^4_{a,b,c,e}=S
}
$$
From Definition \ref{df:Mal_groupoid} the Malgrange-Galois groupoid of a vector field $\vec X$ is given by: 
$$
\mathrm{Mal} (\vec X) = \left\{ \sigma \in \mathrm{Aut}M \mid \mbox{for all } f\in \mathbb C(\mathrm{R}M)^{\vec X}\!,\;\; f \circ \sigma^{(\infty)} = f|_{({\rm R}M)_{s(\sigma)}}\right\}
$$ 

\begin{remark}
\label{rmk:frame_invariants}
From the known invariants of $X$ we can obtain some informations of its Malgrange-Galois groupoid. First, the conserved quantities $a,b,c,e$ are by themselves rational differential invariants of order $0$. Therefore, they span a $\Delta$-subfield\footnote{Let us recall that $\Delta = \{\delta_1,\ldots,\delta_7\}$ stands for the system of total derivative operators with respect to $\varepsilon$'s that give to $\mathbb C({\rm R}M)$ the structure of $\Delta$-field.} 
$\mathbb C\langle a,b,c,e \rangle_{\Delta}$ of the $\Delta$-field $\mathbb C(\mathrm RM)^{\vec X}$ of rational differential invariants of $\vec X$. The $\mathcal D$-groupoid corresponding to such field is the groupoid of formal maps respecting the projection $\bar\rho$. Therefore,
\begin{align*}
\mathrm{Mal}(\vec X) &\subset  \Big\{ \phi \in \mathrm{Aut}(\mathbb C^7) \; \mid
 \; \bar\rho\circ \phi = \bar\rho
\Big\}
\end{align*}
Also, from geometric invariants we can obtain information. There is an intrinsic connection between geometric structures and $\mathcal D$-groupoids. As it is shown (\cite{casale2004groupoide} Theorem 1.3.2 p. 10) any transitive $\mathcal D$-groupoid is the groupoid of invariance of a geometric structure. Let us see some examples. 
\begin{itemize}
    \item We know $\vec Xx =1$ and therefore $\mathrm{Lie}_X(dx) = 0$. Therefore, $1$-form $dx$ can be seen as a rational tensor invariant by $X$, so that
    \begin{align*}
    \mathrm{Mal}(\vec X) &\subset  \Big\{ \phi \in \mathrm{Aut}(\mathbb C^7) \;\mid \; \phi^*(dx) = dx \Big\}
    \end{align*}
    How can we obtain the $\Delta$-field of rational differential invariants associated to $dx$? 
    Note that, in any given frame $\varphi$, the pullback 
    $$\varphi^*(dx)(0) = x_{:\epsilon_1}d_0\varepsilon_1 + \ldots x_{:\epsilon_7}d_0\varepsilon_7$$
    is a co-vector in $\mathbb C^m$. Its coordinates $x_{:\epsilon_1}$, $\ldots$, $x_{:\epsilon_7}$ are the differential invariants defining the $\mathcal D$-groupoid ${\rm Sym}_\Delta(\mathbb C\langle x_{:\epsilon_1},\ldots,x_{:\epsilon_7}\rangle_{\Delta})$ of symmetries of $dx$. 
    
    \item 
    A rational vector field $\vec Y$ in $M$ can be seen as a particular case of a geometric structure. In order to associated to this geometric structure as a $\Delta$-field of rational differential invariants we only need to observe that each $1$-frame $j^1_0\varphi$ at $p\in M$ determines a basis $\{A_1(j_0^1\varphi),\ldots,A_7(j_0^1\varphi)\}$ of $T_pM$ where
    $A_i(j_0^1\varphi) = d_0\varphi\left(\frac{\partial}{\partial \varepsilon_i} \right)_0.$
    The coordinates $Y_i$ of $Y$ in such basis can be seen as rational functions on ${\rm R}_1M$,
    $$Y(p) = \sum_{i=1}^7 Y_i(j_0^1\varphi)A_1(j_0^1\varphi).$$
    The $\mathcal D$-groupoid of symmetries of $\vec Y$ is then ${\rm Sym}_\Delta(\mathbb C\langle Y_1,\ldots Y_7\rangle_\Delta)$. In particular we have $[\vec X, \vec X] = 0$, so that the Malgrange-Galois groupoid of $\vec X$ is included in the $\mathcal D$-groupoid of symmetries of $\vec X$:
    \begin{align*}
    \mathrm{Mal}(\vec X) &\subset  \Big\{ \phi \in \mathrm{Aut}(\mathbb C^7) \;\mid \;\phi_*(\vec X) = \vec X \Big\}.
    \end{align*}
    
    \item The vector field $\vec X$, when restricted to specific values of the parameters, is a non-autonomous Hamiltonian with respect to the form $dq\wedge dp$. It follows that $\mathrm{Lie}_{\vec X} (dq\wedge dp) = d i_X (dp\wedge dq) = d(\frac{\partial H}{\partial q}dq + \frac{\partial H}{\partial p}dp)$ and then  $$ {\rm Lie}_{\vec X}(dp\wedge dq)\equiv 0\ \mathrm{mod}\ dx, da,db,dc,de.$$ This means that the rank $2$ bundle $\mathrm{ker}(d(\pi\circ\rho))$ is endowed with an $\vec X$-invariant volume form, this can also seen as a geometric structure in $\mathbb C^7$, defined by the class of $dp\wedge pq$ modulo $\pi^*\Omega_{B}^1$, and yields the following restriction,
    \begin{align*}
    \mathrm{Mal}(\vec X) &\subset  \Big\{ \phi \in \mathrm{Aut}(\mathbb C^7) \;\mid \; \phi^* (dq\wedge dp) \equiv dq\wedge dp \mod \pi^*\Omega^1_B 
    \Big\}.
    \end{align*}
\end{itemize}

\end{remark}

Summarizing Remark \ref{rmk:frame_invariants}, we have the following restrictions of the Malgrange-Galois groupoid of $X$,
\begin{align}
\label{eq:Mal_subset}
\mathrm{Mal}(\vec X) &\subset  \Big\{ \phi \in \mathrm{Aut}(\mathbb C^7) \;\mid\;\phi_*\vec X=\vec X; \quad \bar\rho\circ\phi = \bar\rho;\nonumber
\\ 
& \quad \quad \quad \phi^* dq\wedge dp \equiv dq\wedge dp \mod  \pi^*\Omega^1_B; 
\quad \phi^* dx =dx
\Big\}
\end{align}

Malgrange-Galois groupoid of Painlev\'e VI equation, with fixed values of parameters, has been found by Cantat--Loray. In what follows it is necessary to distinguish the set \emph{Picard parameters},
the following subset of $\mathbb C^4$:
$$
\Big\{
(a,b,c,e)\in (\textstyle{\frac{1}{2}}+\mathbb Z^4)
\Big\} 
\cup 
\Big\{
(a,b,c,e)\in \mathbb Z^4 \mid a+b+c+e \mbox{ is even }
\Big\} 
$$

\begin{proposition}[ Theorem 6.1 in \cite{cantat2009dynamics}]
\label{pr:picar_param}
For parameters $( a,b,c,e) $ not in Picard parameter set, Malgrange-Galois groupoid is given by:
\begin{align*}
\mathrm{Mal}(\vec X|_{\mathbb C^3 \times \{ ( a,b,c,e)\} } )
&=
\Big\{
\phi : ( \mathbb C^3, *) \overset{\sim}{\longrightarrow} ( \mathbb C^3,\star) \mid \phi^* dx =dx;\,\, \phi_*(\vec X)= \vec X;
\\
& \quad \quad \quad \phi^* dp\wedge dq = dp \wedge dq \mod dx
\Big\}
\end{align*}

In this formula $a,b,c,e$ are fixed parameters. The asterisk and the star stand for arbitrary points in $\mathbb C^3$.
\end{proposition}

As a direct consequence of Theorem \ref{th:specialisation} we have the following:

\begin{proposition}
\label{pr:inclusion_picard_param}
There is an inclusion $\mathrm{Mal} (\vec X\mid_{ \mathbb C^3 \times \{ (a,b,c,e)\}}) 
\subset 
\mathrm{Mal}(\vec X)\mid_{ \mathbb C^3 \times \{ (a,b,c,e)\}}$.
\end{proposition}


\subsection{Transversal part of $\textrm{Mal}(\vec X)$}

Let us consider $\mathrm{Sym}(\mathbb C\langle x \rangle_\Delta)$, the $\mathcal D$-groupoid of formal diffeomorphisms $\phi : (\mathbb C^7,* ) \to (\mathbb C^7, \star) $ that leave invariant the $x$ coordinate, \emph{i.e.} such that $x\circ \phi = x$. We want to prove the equality in Equation  (\ref{eq:Mal_subset}). To do this we look at $\mathrm{Mal}(\vec X) \cap \mathrm{Sym}(\mathbb C\langle x \rangle_\Delta)$. We have that 
\begin{align}
\label{eq:Mal_cap_inv(x)}
\mathrm{Mal}(\vec X) \cap  \mathrm{Sym}(\mathbb C\langle x \rangle_\Delta) &\subset \Big\{ \phi: (\mathbb C^7,*) \to ( \mathbb C^7, \star) \mid \phi^*\vec X = \vec X; \quad \pi\circ \phi = \pi;\nonumber
\\
 & \quad \quad \quad
 \quad \quad \quad  
 \quad \quad \quad 
 \phi^* (dq\wedge dp) = dq\wedge dp \quad\mathrm{mod }\,\, \pi^*\Omega^1_{B}
\Big\}
\end{align}

The plan of the remaining proof is to obtain equality in Equation \eqref{eq:Mal_cap_inv(x)}. To see this we restrict our attention to the $\mathcal D$--Lie algebra 
$$\mathcal L = \mathrm{Lie }\big(\mathrm{Mal}(\vec X)\cap \mathrm{Sym}(\mathbb C\langle x \rangle_\Delta)\big).$$
This Lie algebra is seen to be of the kind described in Kiso--morimoto Theorem \ref{th:kiso-mori} so we apply it. By these means we get the desired equality but for $\mathrm{Lie}\big( \mathrm{Mal}(\vec X)\cap \mathrm{Sym}(\mathbb C\langle x \rangle_\Delta)\big)$ instead of $\mathrm{Mal}(\vec X)\cap \mathrm{Sym}_\Delta(\mathbb C\langle x \rangle_\Delta)$. The final steps will be to ``get rid'' of the terms $\mathrm{Lie}$ and $\mathrm{Sym}(\mathbb C\langle x \rangle_\Delta)$.

\begin{proposition}
The $\mathcal D$--Lie algebra  
$\mathcal L$
and $\pi\colon \mathbb C^7\to\mathbb C^5$ are under the hypothesis of Theorem \ref{th:kiso-mori}, namely: first, that  $\mathcal L$ is tangent to $\pi$ and  second, that
there exists a form $\omega\in \Omega^m_M$ such that for all $s\in S$, $\omega|_{M_s}$ is a non--identically zero $m$--form on $M_s$ satisfying 
$$ 
\label{condition on mathcal L in Malgrange chapter}
	\mathcal L|_{M_s}
	=\big\{ 
	v \mbox{ vector fields on } M_s \colon
	\mathrm{Lie}_v\, \omega|_{M_s} =0
	\big\}.$$
\end{proposition}

\begin{proof} 
First, $\mathcal L$ is tangent to the fibers of $\pi : \mathbb C^7 \to \mathbb C^5$. Second, it preserves the volume form $dq\wedge dp$ on the fibers. Let us compute the restriction of $\mathcal L$ to a fiber $M_s$ with $s = (x_0,a,b,c,e)$. Let us assume that $(a,b,c,e)$ is not in the Picard parameter set. Then we by Proposition n \ref{pr:picar_param}
we have that ${\rm Mal}(\vec X|_{\bar\rho^{-1}(a,b,c,e)})$ is $\mathcal D$-groupoid of symmetries of $\vec X|_{\bar\rho^{-1}(a,b,c,e)}$,
$dx$ and the class of $dq\wedge dp$ mod $dx$. By corollary  \ref{co:restriction_Mal_fiber} it coincides with
${\rm Mal}(\vec X)|_{\bar\rho^{-1}(a,b,c,e)}$. 

Taking the intersection with ${\rm Sym}_\Delta(\mathbb C\langle x \rangle_\Delta)$ we obtain that the restriction of
${\rm Mal}(\vec X)\cap {\rm Sym}_\Delta(\mathbb C\langle x \rangle_\Delta)$ to $\bar\rho^{-1}(a,b,c,e)$ 
is the the $\mathcal D$-groupoid of symmetries of $\vec X|_{\bar\rho^{-1}(a,b,c,e)}$, $x$, and the class of $dq\wedge dp$ mod $dx$. It follows that 
$\mathcal L|_{\bar\rho^{-1}(a,b,c,e)}$ is the 
$\mathcal D$-algebroid of infinitesimal symmetries of $\vec X|_{\bar\rho^{-1}(a,b,c,e)}$, $x$ and the class of $dq\wedge dp$ mod $dx$. When we fix the value of $x = x_0$ then we obtain that 
$\mathcal L|_{M_s}$ is the $\mathcal D$-Lie algebroid of infinitesimal symmetries of $dq\wedge dp$.
\end{proof}

Therefore, from
Kiso--Morimoto Theorem \ref{th:kiso-mori}, we have that:
\begin{enumerate}
	\item[(a)] There exists  a foliation $ \mathcal F$ on $ \mathbb C^5$,  
    \item[(b)] there exists $ \mathcal H$ on $ \mathbb C^7$ such that $ d\pi( \mathcal H) = \mathcal F$, ${\rm rank}(\mathcal H) = {\rm rank}(\mathcal F)$,
    \item[(c)]
    $\mathcal L$ is the $\mathcal D$-Lie algebra of vector fields tangent to the fibers of $\pi$, preserving the volume form $dp\wedge dq$, and preserving the $\mathcal F$-connection\footnote{Let us recall that if $M\to B$ is a bundle and $\mathcal F$ is a foliation in $B$, then a $\mathcal F$-connecion is a linear sub-bundle of ${\rm T}M$ of the same rank that $\mathcal F$ and that projects onto $\mathcal F$, see \cite{malgrange2010pseudogroupes}.}
    $\mathcal H$.
    \begin{align*}
\mathcal L & = \Big\{ v\in J({\rm T}M/M) \mid d\pi(v) =0; \;\; \mathrm{Lie}_v \omega = 0\mod \pi^* \Omega^1_S;
\\
& \qquad \qquad \forall \vec Y \text{ tangent to }\mathcal H \quad [v,\vec Y]\ \text{ is tangent to } {\mathcal H}
\Big\}
\end{align*}
\end{enumerate}

In particular, the vector field $\vec X$ is tangent to $\mathcal H$. Therefore $\pi_*(\vec X) = \frac{\partial}{\partial x}$ is tangent to $\mathcal F$. Thus, the foliation $\mathcal F$ is $\rho$-projectable on a foliation of $\mathbb C^4$ of the same codimension. Let us denote by $\overline{\mathcal F}$ this projection.


\subsection{Affine Weyl group}

Let us note that if $\phi$ is a birational automorphism of $M$ and $\vec X$ is a rational vector field on $M$, then the pullback by the extension, $\phi^{(\infty)}\colon \mathrm{R}M \to \mathrm{R}M$ sends rational differential invariants of $\phi_*(\vec X)$ to rational differential invariants of $\vec X$. If follows clearly that $\phi$ induces an isomorphism between $\mathrm{Mal}(\vec X)$ and $\mathrm{Mal}(\phi_*(\vec X))$. In particular, if $\phi$ is a \emph{discrete birational symmetry} \index{Discrete birational symmetry}  of $\vec X$, that is, a birational automorphism of $M$ such that $\phi_*(\vec X) = \vec X$, then $\phi$ leaves $\mathrm{Mal}(\vec X)$ invariant. That is,
$$
j_p\sigma \in \mathrm{Mal}(\vec X) \Longleftrightarrow j_{\phi(p)}(\phi\circ \sigma \circ \phi^{-1})\in \mathrm{Mal}(\vec X).
$$

\begin{lemma}
Foliation $\overline{\mathcal F}$ in $\mathbb C^4$ is regular.
\end{lemma}

\begin{proof} 
Equation \eqref{eq:PainVI_vec_field2} admits a discrete group $\widetilde{W}$ of symmetries, known as Backl\"und transformations, isomorphic to the extended affine Weyl group \index{Affine Weyl group} $D_4^{(1)}$. All elements of this group are birational transformations of $\mathbb C^7$. In particular this group contains a subgroup $G_4$ isomorphic to $\mathbb Z^
4$ with the following characteristics:
\begin{itemize}
    \item[(a)] The action of $G_4$ in $\mathbb C^7$ is projectable by $\bar\rho$ to the action of translations with integer displacements in $\mathbb C^4$.
    \item[(b)] The function $x$ is invariant by the action. 
\end{itemize}
This  group of translations is listed in \cite[p.6]{noumi2002new}.\\

The action of $ G_4$ preserves $\mathrm{Mal}(\vec X)$. The function $x$ is an invariant for this action, then $ \mathrm{Sym}(\mathbb C\langle x \rangle_\Delta)$ is also preserved. It follows that $G_4$ leaves $ \mathcal L$ invariant, and its defining equations, in particular, for $\phi \in G_4$, $\phi_*( \mathcal H) = \mathcal H$.
Denote by $\bar \phi$ the projection of $\phi$ as a birational transformation of $ \mathbb C^4$ and observe $\bar \phi$ is a translation. Then $\bar \phi$ sends $\overline{ \mathcal F}$ into $\overline {\mathcal F}$, \emph{i.e.} $d \phi( \overline {\mathcal F})=\overline {\mathcal F}$.  As the set of singularities of $\overline {\mathcal F}$ is  a proper Zariski closed  of $\mathbb C^4$, invariant under translation by $\mathbb Z^4$,  it must be empty. 
\end{proof}




A result of Iwasaki, \cite[Theorem 1.3]{iwasaki2008finite}  implies the following:

\begin{lemma} 
The solutions to $P_{VI}$ with finite monodromy are algebraic.
\end{lemma}

This allows us to see:

\begin{lemma}
$\overline{\mathcal F}$ is a foliation by points.
\end{lemma}

\begin{proof}
In \cite{lisovyy2014algebraic} a list of all possible algebraic solutions of Painlev\'e VI equation is presented. These appear at special values of the parameters. The solution numbered as 45 in \cite[p.52]{lisovyy2014algebraic} is algebraic with 72 branches, and happens at the parameter $\theta = (1/12,1/12,1/12,11/12)$. 
If the dimension of $ \overline{\mathcal F}$ were greater than $0$, we could find a path on a leaf of $ \overline{\mathcal F}$ along which we could extend the given solution to a solution at another parameter, with also $72$ branches.  By Iwasaki, this solution is algebraic. But we know the original solution is unique, all other algebraic solutions have less than $72$ branches. We conclude that the foliation cannot have dimension greater than 1 so it must be a foliation by points.
\end{proof}

Therefore $\mathcal F = \left\langle
\frac{\partial}{\partial x}
\right\rangle$ and $\mathcal H = \left\langle \vec X\right\rangle$. Note that, for a vector field $v$ tangent to $\bar\rho$, as $X$ is transversal to $\bar\rho$, it is equivalent to say that $\mathrm{Lie}_v\mathcal H \subset \mathcal H$ or $[v,X]=0$. This completes the proof of the following result.

\begin{proposition}
\label{prop:first_description_of_mathcal_L}
\begin{align*}
\mathcal L
&= \Big\{ \vec{v} \text{ vector field on } \mathbb C^7 \;\mid\;
\vec v\cdot x=\vec v\cdot a=\vec v \cdot b=\vec v\cdot c =\vec v\cdot e =0; 
\;\; [\vec v,X]=0; \\
& \qquad \qquad \mathrm{Lie}_{\vec v} (dp\wedge dq) =0 \mod \pi^\ast \Omega_B^1
\Big\}
\end{align*}
\end{proposition}

By integration of the $\mathcal D$-Lie algebra $\mathcal L$ we obtain:

\begin{proposition}
\label{pr:description_of_Mal_cap_Inv}
\begin{align}
\label{eq:description_of_Mal_cap_Inv}
\mathrm{Mal}(\vec X) \cap \mathrm{Sym}(\mathbb C\langle x \rangle_\Delta) 
&= \Big\{ \phi \mid
\pi \circ \phi = \pi; \;\; \phi_*(\vec X)=\vec X;\nonumber
\\
& \qquad \qquad \phi^*(dp\wedge dq) \equiv dp\wedge dq \mod \pi^\ast \Omega_B^1
\Big\}
\end{align}
\end{proposition}

\begin{proof}
The Lie algebra $ \mathcal L$ is determined by $ \mathcal L_1$, its first order part. Let us call $\mathcal G$ the right hand side of (\ref{eq:description_of_Mal_cap_Inv})$; \mathcal G$ is determined by $ \mathcal G_1$, its first order part.The Malgrange-Galois groupoid $ \mathrm{Mal}(\vec X)$ is determined by $ \mathrm{Mal}_1(\vec X)$ too.  It can be proved that $ \mathcal G_1$ is connected with respect to source and target. Then $ \mathcal G_1$ is the least Lie subgroupoid of $ \mathrm{Aut}_1( \mathbb C^7)$ such that $ \mathrm{Lie}( \mathcal G_1) = \mathcal L_1$. By (\ref{eq:Mal_cap_inv(x)}) we know that $ \mathrm{Mal}_1(X) \subset \mathcal G_1$. As  $ \mathrm{Lie}( \mathrm{Mal}(\vec X)) = \mathcal L_1$ then $ \mathrm{Mal}_1(X) = \mathcal G_1$. We conclude that $  \mathrm{Mal}(\vec X) = \mathcal G$.
\end{proof}

\begin{theorem}
\label{th:Final_Mal_groupoid_PVI}
The Malgrange-Galois groupoid of Painlev\'e VI equation is given by
\begin{align}
\label{eq:Mal_final_description}
\mathrm{Mal}(\vec X)  
&= \Big\{ \phi \mid
\bar\rho \circ \phi = \bar\rho; \;\; \phi^*(dx) = dx; \;\; \phi_*(\vec X)=\vec X;\nonumber
\\
& \qquad \qquad \phi^*(dp\wedge dq) \equiv dp\wedge dq \mod\pi^\ast \Omega_B^1
\Big\}
\end{align}
\end{theorem}

\begin{proof}
We already have pointed in (\ref{eq:Mal_subset}) that $ \mathrm{Mal}(\vec X)$ is contained in the right side set of Equation \eqref{eq:Mal_final_description}. Let us prove the remaining inclusion. Fix an integer $k$ large enough and let $\psi\colon (\mathbb C^7,z_0)\to
(\mathbb C^7,z_1)$ be a map that satisfy equations in the right hand side of \eqref{eq:Mal_final_description}. We are going to show that $j_{z_0}^k\psi$ is in $ \mathrm{Mal}_k(\vec X)$. Let $ \varepsilon = x(z_1) -x(z_0)$. Without loss of generality, assume that $ \exp( - \varepsilon \vec X)$ is defined in a neighborhood such that if $z_2 = \exp( - \varepsilon \vec X)(z_1)$ then all the points $z_0,z_1$ and $z_2$ are inside an open set where $ \mathrm{Mal}(\vec X)$ is effectively a groupoid. Let the map $\phi\colon (\mathbb C^7,z_0)\to (\mathbb C^7,z_2)$ be defined by $\phi = \mathrm{exp}(-\varepsilon \vec X)\circ \psi$. Observe that $\phi$ respects $\bar \rho$, $dx$, $dp\wedge dq|_{ \mathbb C^2_{p,q}}$, and the field $\vec X$. Moreover, since $\phi^*(dx) = dx$ we have $\phi^*(x) = x + \lambda$ with $\lambda$ a constant, indeed $x(\phi(z_1)) = x(z_1)$ so that $\lambda = 0$ and $\phi^*(x) = x$. It follows that $\phi$ preserves $\pi$. By Proposition \ref{pr:description_of_Mal_cap_Inv} $j_{z_0}^k \phi$ is in $ \mathrm{Mal}_k(X)$. Then $j_{z_0}^k \psi = j_{z_2}^k \exp ( \varepsilon X) \circ j_{z_0}^k\phi$ is in $ \mathrm{Mal}_k (X)$. 
\end{proof}


\begin{corollary}\label{co:par}
If $y(x,a,b,c,e)$ is a parameter dependent solution of the sixth Painlev\'e equation then its annihilator in $\mathcal O_{J(\mathbb C^5, \mathbb C)}$ is the $\partial$-ideal generated by the sixth Painlev\'e equation.  
\end{corollary}

Using the Hamiltonian formulation, we want to prove that the Zariski closure of a parameters dependant solution $p(x,a,b,c,e) ; q(x,a,b,c,e)$ in $J(\mathbb C^5, \mathbb C^2)$ is $V$ : the subvariety defined by the differential ideal generated by $\partial_x p - \frac{\partial H}{\partial q}$ and $\partial_x q + \frac{\partial H}{\partial p}$.\\

\begin{proof}

To prove this, one studies the action of $\mathrm{Mal}(\vec X)$ on $\partial_x^{tot}$. This action is well defined because the elements of $\mathrm{Mal}(\vec X)$ are fiber preserving transformations of $\mathbb C^5 \times \mathbb C^2 \to \mathbb C^5$. Let $W\subset V$ be the Zariski closure of a parameter dependant solution.\\

The proof follows from the next two lemmas.

\begin{lemma}\label{lm:W}
$\mathrm{Mal}(\vec X)$ preserves $W$
\end{lemma}

\begin{remark} The restriction of $\partial_x^{tot}$ on $J_0(V) = \mathbb C^7$ is the vector field $\vec X$. As $V$ or $W$ are define by $\partial_x^{tot}$-ideals, the extension of $\vec X$ on $J(\mathbb C^5, \mathbb C^2)$ is tangent to $V$ and to $W$. The stabilizer of $W$ is a $\mathcal D$-groupoid containing the flow of $\vec X$: it must contain $\mathrm{Mal}(\vec X)$.
\end{remark}

\begin{lemma}\label{lm:WV}
$\mathrm{Mal}(\vec X)$ acts transitively on $V$ over the parameter space.
\end{lemma}

This means that fibers $\mathrm{Mal}(\vec X)_{a,b,c,e}$ act transitively on fibers $V_{a,b,c,e}$.

\begin{proof}
Fix parameters $(a_0,b_0,c_0,e_0)\in \mathbb C^4$.
Let $(p_1(x,a,b,c,e), q_1(x,a,b,c,e))$ be a germ of solution on $(\mathbb C, x_1) \times (\mathbb C^4, (a_0,b_0,c_0,e_0))$ and  $(p_2(x,a,b,c,e), q_2(x,a,b,c,e))$ be a germ of solution on $(\mathbb C, x_2) \times (\mathbb C^4, (a_0,b_0,c_0,e_0))$. One can find family of translations parameterized by $(a,b,c,e)$ : $\mathbb C^2 \times \mathbb C^4 \to \mathbb C^2 \times \mathbb C^4$ sending $(p_1(x_1,a,b,c,e), q_1(x_1,a,b,c,e))$ on $(p_2(x_2,a,b,c,e), q_2(x_2,a,b,c,e))$. Using trajectories of $\vec X$, one can extent this map to neighborhoods of fibers $\{x=x_1\}$ and $\{x=x_2\}$ above translation in the variable $x$ as a map preserving $\vec X$.
As Painlev\'e equation preserves $dp\wedge dq\wedge dx \ \mod da, db , dc, de$, this extension also. It satisfies all conditions and belongs to $\mathrm{Mal}(\vec X)_{a,b,c,e}$
\end{proof}

Finally, by Lemmas \ref{lm:W} and \ref{lm:WV} we have $V = W$ and thus the graph of the parameter dependent solution in $J(\mathbb C^5,\mathbb C^2)$ is Zariski dense in the variety defined by the radical $\partial$-ideal generated by the sixth Painlevé equation.
\end{proof}

\appendix

\section{Kiso--Morimoto theorem for $\mathcal D$-Lie algebras}\label{s:kiso}

A $\mathcal D$-Lie algebra is, outside of its singularities, a continuous Lie algebra sheaf of vector fields. A result from Kiso (Theorem 5.1 in \cite{kiso1979local} case II).  based on a Lemma that he attributes to Morimoto (Proposition 4.1 in \cite{kiso1979local}) allows to classify all continuous Lie algebra sheaves acting transitively on the fibers of a bundle and leaving an invariant volume form on the fibers  Here we present an algebraic version of Kiso-Morimoto theorem, the proof is the same as in the original paper \cite{kiso1979local}, but reasoning on the generic point. Kiso's original statement is only concerned with the canonical form of the $\mathcal D$-Lie algebra. However, in our statement we make explicit its realization as a $\mathcal D$-Lie algebra fixing some singular foliation transversal to the fibers. The hypothesis of Kiso-Morimoto apply to some $\mathcal D$-Lie algebra related to the Malgrange-Galois groupoid. \\

Let us set up the main elements relevant to the statement. Let us consider a regular submersion $\rho\colon M\to S$ between irreducible affine varieties. We also set, along this section $m = \dim M - \dim S$. For each $s\in S$ let us denote by $M_s$ its fiber by $\rho$. Also, let us denote by $\mathcal P = {\rm ker}(d\rho) \subset {\rm T}M$ the foliation whose leaves are the fibers of $\rho$ and by $\mathfrak X_{\mathcal P}$ the Lie algebra of rational vector fields tangent to $\mathcal P$. Note that $\mathcal P^\perp = \rho^*{\rm T}M \subset {\rm T}^*M$ is the bundle of $1$-forms vanishing on $\mathcal P$. Second, assume that $\mathcal L$ is a $\mathcal D$-Lie algebra such that for all jet of vector field $\vec X\in L$ at any $p\in M$ is $d\rho(X) = 0$. Then, for each $s\in S$ we may consider the restriction 
$$
\mathcal L|_{M_s} = \{ \vec X|_{M_{\rho(p)}} \,\colon\, \vec X \in \mathcal L \mbox{ based at } p\in M \mbox{ with } \rho(p)=s\},
$$ 
which is a $\mathcal D$-Lie algebra in $M_s$.\\

Let us also recall that a linear sub-bundle $\mathcal H\subset {\rm T}M$ is $\rho$-projectable if there is a sub-bundle $\mathcal F\subset {\rm TS}$ such that $d\rho(\mathcal H_p) = \mathcal F_{\rho(p)}$ for all $p\in M$. In such case $\mathcal F$ is completely determined by $\mathcal H$ and we write $\mathcal F = \rho_*(\mathcal H)$. This includes the case of foliations, that we see as linear sub-bundles of the tangent bundle.

\begin{theorem}[Kiso--Morimoto] \label{th:kiso-mori}
Let $\rho\colon M\to S$ be a regular submersion between irreducible affine varieties. Let $\mathcal L$ be a $\mathcal D$-Lie algebra in $M$ such that:
\begin{enumerate}
	\item[(Hyp1)] 
	$\mathcal L$ is tangent to $\rho$, meaning that for any formal vector field $\vec X\in \mathcal L$, $d\rho(\vec X) = 0$.
    \item[(Hyp2)]
    There exists a rational $m$-form $\omega\in \Omega^m_M$ such that for generic $s\in S$, $\omega|_{M_s}$ is not zero and satisfies\footnote{Note that if $\vec X\in J({\rm T}M/M)$ is a jet of vector field in $M$ and $\omega$ is an analytic form then ${\rm Lie}_{\vec X}\omega$ is well defined as a jet of form. The Lie derivative of a rational form is also well defined at the points outside its domain a a local meromorphic form by considering the rational form as linear combination of analytic forms with rational coefficients.}
    \begin{equation} \label{condition on mathcal L}
		\mathcal L|_{M_s}
		=\big\{ 
		\vec X \in J({\rm T}M_s/M_s) \mbox{ such that }
		{\rm Lie}_{\vec X}\, \omega|_{M_s} =0
		\big\}.
    \end{equation}
\end{enumerate}
Then there exists a singular foliation $\mathcal H$ over $M$ such that 
$$
\mathcal L  = \Big\{ \vec X\in J({\rm T}M/M) \mid d\rho(\vec X) =0, \,\, \mathrm{Lie}_{\vec X} \omega \equiv 0 \,\, \mathrm{mod} \,\, \rho^* \Omega^1_S,$$ 
\vspace{-6mm}
\begin{equation}\label{eq:km}
 \quad\quad \quad\forall \vec Y\in \mathfrak X_{\mathcal H} \,\, [\vec X,\vec Y]\in \mathfrak X_{\mathcal H}
\Big\}.
\end{equation}
moreover $\mathcal H$ is $\rho$-projectable and ${\rm rank}(\mathcal H) = {\rm rank}(\rho_*\mathcal H)$.
\end{theorem}

\begin{remark}
By restriction to a suitable Zariski open subset $U\subset M$ we may:
\begin{itemize}
    \item[(a)] Assume that for all $k\in \mathbb N$, $\mathcal L_k\to M$ is a regular linear bundle. 
    \item[(b)] Assume that $\omega|_{M_s}$ is a regular volume form for all $s\in S$.
\end{itemize}
\end{remark}

The proof will take several steps. First, we shall find the foliation $\mathcal H$, that appears as the linear isotropy of $\mathcal L$. Then we will consider $\overline{\mathcal L}$ the $\mathcal D$-Lie algebra given in equation \eqref{eq:km}. Then we will check if $\overline{\mathcal L} = \mathcal L$. Before the exposition of the proof, we need some preliminary considerations.

\subsection{On the $\mathcal D_M$-module $N = {\rm ann}(\mathcal L)$} Let $\Omega^1_{M/S}$ be the $\mathbb C(M)$-space of rational $1$-forms restricted to the fibers of $\rho$, \emph{i.e.} rational sections of $\mathcal P^*$. We have a natural exact sequence.
$$
\xymatrix{ 
0 \ar[r] &\mathbb C(M)\tens_{\mathbb C(S)} \rho^*\Omega_S^1 \ar[r] & \Omega_M^1 \ar[r] & \Omega_{M/S}^1 \ar[r] & 0
}$$
 The $\mathcal D_M$-module $\mathcal D_M\tens_{\mathbb C(M)}\rho^*\Omega1^1_S$ is the $\mathcal D_M$-module of differential operators from $\mathfrak X_M$ onto $\mathbb C(M)$ vanishing on vector fields tangent to the fibers of $\rho$. By (Hyp1) in Theorem \ref{th:kiso-mori} we have that $\mathcal D_M\tens_{\mathbb C(M)}\rho^*\Omega^1_S\subseteq N$. Thus, by taking tensor product with $\mathcal D_M$ and restricting the above sequence to $N$ we obtain exact sequences of $\mathcal D_M$-modules:
$$
\xymatrix{ 
0 \ar[r] &\mathcal D_M \tens_{\mathbb C(M)} \rho^*\Omega_S^1 \ar[r] & 
{\rm Diff}(\mathfrak X_M,\mathbb C(M)) \ar[r]^{\mathrm{rst}} & 
{\rm Diff}(\mathfrak X_{\mathcal P},\mathbb C(M)) \ar[r] & 0 \\
0 \ar[r] &\mathcal D_M \tens_{\mathbb C(M)} \rho^*\Omega_S^1 \ar@{}[u]|-*[@]{=} \ar[r] & N \ar@{}[u]|-*[@]{\subset}
\ar[r]^{\mathrm{rst}} & {\rm rst}(N) \ar@{}[u]|-*[@]{\subset} \ar[r] & 0
}
$$ 
where ${\rm rst}$ represents the restriction of differential operators defined in $\mathfrak X_M$ to the subspace $\mathfrak X_{\mathcal P}$. As $N$ contains the kernel of ${\rm rst}$ it follows that $\mathrm{rst}^{-1}(\mathrm{rst}(N)) = N$. Let $\mathcal D_{M/S}\subset \mathcal D_M$ be the ring generated by $\mathbb C(M)$ and $\mathfrak X_{\mathcal P} $, consisting of rational differential operators that are tangent to the fibers of $\rho$. Then, we have:
$$\mathcal D_{M/S}\tens_{\mathbb C(M)} \Omega^1_{M/S} \subset {\rm Diff}(\mathfrak X_\mathcal P, \mathbb C(M)).$$
Now, let us define $\overline{N} = {\rm rst}(N) \cap (\mathcal D_{M/S}\tens_{\mathbb C(M)} \Omega^1_{M/S})$. This $\mathcal D_{M/S}$-module $\overline{N}$ is the system of linear partial differential equation satisfied by the restriction of $\mathcal L$ to a generic fiber $M_s$. \\

Let us consider now the invariant $m$-form $\omega$. Let $\bar\omega = \omega|_{\mathcal P}$ which is a regular section of $\Omega^m_{M/S}$. For each $\vec X\in \mathfrak X_{\mathcal P}$ we define the value of the divergence operator 
$${\rm div}_{\bar \omega}(\vec X) = \frac{{\rm Lie}_{\vec X}\bar\omega}{\bar\omega} \in \mathbb C(M).$$
The divergence operator is a first order differential operator 
${\rm div}_{\bar\omega}\in \mathcal D_{M/S}\tens_{\mathbb C(M)}\Omega_{M/S}^1$. We can reformulate (Hyp2) in Theorem \ref{th:kiso-mori} as $\overline N = {\mathcal D}_{M/S}{\rm div}_{\bar\omega}$, it follows: 
$${\rm rst}(N) \supseteq \mathcal D_M\overline{N} = \mathcal D_M{\rm div}_{\bar\omega}.$$

\subsection{Symbol and linear isotropy of $\mathcal L$} Here we will see how to recover the foliation $\mathcal H$ from the linear isotropy of $\mathcal L$.

\begin{definition}\label{df:gk}
For $p\in M$ let $ \mathcal L^{\geq k}_p= \mathcal L_p \cap \mathfrak m_p^k J(TM/M)_p$ be the space of formal vector fields in $\mathcal L_p$ that vanish up to order $k$ at $p$\footnote{Here $\mathfrak m_p$ is the maximal ideal of the local ring of formal developments of functions at $p$}. Then,
$$\mathcal L^{\geq k} = \bigcup_{p\in M} \mathcal L_p^{\geq k}$$
is a linear sub-bundle of $\mathcal L$. Let us define, 
$$\mathfrak g^k = \mathcal L^{\geq k+1}/ \mathcal L^{\geq k+2} \quad \mathfrak g = \bigoplus_{k=0}^\infty \mathfrak g^k;$$ 
each $\mathfrak g^k$ is a finite rank regular linear bundle over $M$, so-called \emph{$k$-th symbol} of $\mathcal L$, and $\mathfrak g$ is a graded linear bundle, so-called the \emph{symbol} of $\mathcal L$. As $\mathfrak g^0$ consists of linear parts of vector fields around fixed points, it is also termed the \emph{linear isotropy of} $\mathcal L$.
\end{definition}

By definition and Kuranishi's fundamental identification $\mathfrak m_p^{k+1}/\mathfrak m_p^{k+2} \simeq {\rm S}^{k+1}{\rm T}_p^*M$ we have that $\mathfrak g^k$ is 
a linear sub-bundle  of $ {\rm S}^{k+1} {\rm T}^*M \otimes \mathcal P$ described by the principal symbol of the differential operators of order $(k+1)$ of $N$. For $k = 0$ we have: 
$$\mathfrak g^0\subset {\rm T}^*M \otimes \mathcal P = {\rm Lin}({\rm T}M,\mathcal P) \subset \mathrm{End}({\rm T}M).$$
Note that the Lie bracket of linear vector fields coincides with the bracket of endomorphisms. From being $\mathcal L$ a Lie-algebra bundle, it follows that $\mathfrak g^0$ is a Lie algebra sub-bundle of ${\rm End}({\rm T}M)$. It consists of the linear parts of formal vector fields of $\mathcal L$ around a fixed point. Therefore $\mathfrak g^0$ is termed \emph{linear isotropy} of $\mathcal L$.

\begin{lemma}
Let us assume that there is a foliation $\mathcal H$ in $M$ such that equation \eqref{eq:km} holds. Necessarily $\mathcal H \subseteq \ker \mathfrak g^0$.
\end{lemma}
\begin{proof} 
Let us consider $p\in M$, $\vec Y\in \mathfrak X_\mathcal P$, and $v\in \mathcal L$ that vanishes at $p$ with linear part $v^1\in {\rm T}^*_pM\otimes \mathcal P_p$ . In a system of local coordinates $\{x_i\}$  vanishing at $p$ we have expressions:
$$\vec Y = Y_i \frac{\partial}{\partial x_i} + o(x), \quad v = \sum_{i,j} v_j^i x_i \frac{\partial}{\partial x_j} + o(x^2)$$
and then,
\begin{align*}
0 = [ \vec Y,v] = \vec Y\circ v - v\circ \vec Y = Y_iv_j^i \frac{\partial}{\partial x_j} + o(x).
\end{align*}
By evaluating at $p$ we get $0 = \sum_i v_j^i Y_i$ and this implies $\vec Y_p \in \ker v^1$.
\end{proof}

\subsection{Structure of the linear isotropy}  Let us consider the restriction map,
$${\rm rst|_{\mathcal P}}\colon {\rm Lin}({\rm T}M, \mathcal P) = {\rm T}^*M \otimes \mathcal P \to \mathcal P^*\otimes \mathcal P = {\rm End}(\mathcal P),$$
note that this restriction is also a Lie algebra bundle morphism (with the commutator of endomorphisms, or equivalently with the Lie bracket of linear parts of vector fields). The kernel of the restriction ${\rm rst|_{\mathcal P}}$ is the bundle ${\mathfrak b = \mathcal P^\perp \otimes \mathcal P}$ of linear maps from ${\rm TM}$ to $\mathcal P$ vanishing along $\mathcal P$. We have an exact sequence,
$$\xymatrix{
0 \ar[r] & \mathfrak b \ar[r] & {\rm Lin}({\rm T}M,\mathcal P) \ar[r] & {\rm End}(\mathcal P) \ar[r] & 0
}$$
Note that for any elements $b_1,b_2\in \mathfrak b$ we have $b_1\circ b_2 = 0$ (as endomorphisms of ${\rm T}_pM$ for some $p\in M$) and therefore  $\mathfrak b$ is an \emph{abelian} Lie algebra bundle. Since $\mathfrak b$ consists on linear maps vanishing along $\mathcal P$ and $\mathcal P^{\perp} = \rho^*{\rm T}^*S \subset {\rm T}^*M$ we obtain $\mathfrak b = \rho^*{\rm T}^*S \otimes \mathcal P = {\rm Lin}(\rho^*{\rm T}^*S,\mathcal P)$.

\begin{lemma}
The image $\mathfrak g^0$ in ${\rm End}(\mathcal P)$ is the Lie algebra bundle of trace free endomorphisms of $\mathcal P$. That is:
$$\mathfrak g^0|_{\mathcal P} = \mathfrak{sl}(\mathcal P).$$
\end{lemma}
\begin{proof} 
Let's us fix $p\in M$ and see that $\mathfrak g_p^0|_{\mathcal P_p} = \mathfrak{sl}(\mathcal P_p)$. Let $s$ be $\rho(p)$.
\begin{itemize}
    \item[(a)] $\mathfrak g_p^0|_{\mathcal P_p} \subseteq \mathfrak{sl}(\mathcal P_p)$. 
    Let $v^1 \in \mathfrak g_p^0$. Let $\vec X\in \mathcal L_p$ be a formal vector field with base point at $p$ whose linear part is $v^1$. Then $v^1|_{\mathcal P_p}$ is the linear part of $\vec X|_{M_s}$. By (Hyp2) in Theorem \ref{th:kiso-mori} we have that $\vec X|_{M_s}$ is a divergence free formal vector field with respect to the volume $\omega$. Therefore, its linear part $v^1|_{\mathcal P_p}$ is trace free.
    \item[(b)] $\mathfrak{sl}(\mathcal P_p) \subseteq \mathfrak g_p^0|_{\mathcal P_p}$. Let us fix $w^1\in {\rm sl}(\mathcal P_p)$. By (Hyp2) in Theorem \ref{th:kiso-mori} there is a formal vector field $\vec Y \in (\mathcal L|_{M_s})_p$ vanishing at $p$ and whose linear part is $w^1$. Then, there is $\vec X\in \mathcal L_p$ such that $\vec X|_{M_s} = \vec Y$. The linear part of $\vec X$ at $p$ is in $\mathfrak g^0$ and its restriction to $\mathcal P_p$ is $w^1$.
\end{itemize}
\end{proof}

Define $\mathfrak a = \mathfrak b\cap \mathfrak g^0$. we have an exact sequence:
$$
\xymatrix{ 
0 \ar[r] & \mathfrak a \ar[r] & \mathfrak g^0 \ar[r] & \mathfrak h \ar[r] & 0
}
$$
with $ \mathfrak a \subset \mathfrak b$.

\begin{lemma}
$\mathfrak{sl}(\mathcal P)$ acts on $\mathfrak b$ by left composition with the properties:
\begin{itemize}
    \item[(a)] For all $g\in \mathfrak g^0_p$, $b\in \mathfrak b_p$, $[g,b] = g|_{\mathcal P_p}\circ a.$
    \item[(b)] The action of $\mathfrak{sl}(\mathcal P)$ in $\mathfrak b$ preserves $\mathfrak a$.
\end{itemize}
\end{lemma}

\begin{proof} 
For $h\in \mathfrak{sl}(\mathcal P_p)$ and $b\in \mathfrak b_p$. Let $\tilde h\in \mathfrak g^0_p$ be any representative such that $\tilde h|_{\mathcal P_p} = h$. Since the image of $b$ is contained in $\mathcal P_h$ we have:
$$[\tilde h,b] = \tilde h\circ b - b \circ \tilde h = \tilde h \circ b = h\circ b.$$
Finally, in the above equation, if $b\in \mathfrak a_p$ then $b\in \mathfrak g^0$, and then
$[\tilde h,a] = h\circ a\in \mathfrak g^0_p$ we also have $h\circ a \in \mathfrak b_p$ therefore
$h\circ a \in \mathfrak g_p^0\cap b_p = \mathfrak a_p.$
\end{proof}

The following result is Morimoto's Lemma (Proposition 4.1. \cite{kiso1979local} or Lemma 9.1. in \cite{morimoto1977intransitive}). Since we additionally state the rationality of the bundle, we include the proof.

\begin{lemma}\label{lm:morimoto}
There is a rational linear sub-bundle $\mathcal A\subset \rm TM$ such that $\mathfrak a = \mathcal A^\perp \otimes \mathcal P.$
\end{lemma}
\begin{proof}
If $\mathfrak a = 0$ the lemma is true. Let us assume $\mathfrak a\neq 0$. Along the proof of the lemma we replace $M$ by an affine open subset, or by a finite covering of an affine open subset whenever we need. At the end we will check that the linear bundle $\mathcal A$ of the statement is well defined over $M$. Then, we may assume that $\mathfrak a$ decomposes as direct sum of $\mathfrak{sl}(\mathcal P)$-invariant irreducible bundles, 
$$\mathfrak a = \mathfrak a_1 \oplus \ldots \oplus \mathfrak a_r$$
all of them different from $0$. For each $j = 1,\ldots,r$ let us consider the internal contraction:
$${\rm int}^{(j)}\colon \mathfrak X_M \to {\Gamma}^{\rm rat}({\rm Lin}(\mathfrak a_j, \mathcal P)).$$
Here ${\rm int}^{(j)}(\vec X)\colon a\mapsto a(\vec X)$ where $a$ is a section of linear maps from ${\rm T}M$ to $\mathcal P$. The internal contraction is compatible with the $\mathfrak{sl}(\mathcal P)$-action and therefore for any $\vec X\in \mathfrak X_M$
${\rm int}^{(j)}(\vec X)\colon \Gamma^{\rm rat}(\mathfrak a_1) \to \Gamma^{\rm rat}(\mathcal P)$ is 
$\Gamma^{\rm rat}(\mathfrak{sl}(\mathcal P))$-equivariant. 
By Schur's Lemma the dimension of the image of ${\rm int}^{(j)}$ is either $0$ or $1$. If it is $0$, then it follows that $\mathfrak a_j= 0$ which is in contradiction with our hypothesis. Therefore we have that the image of ${\rm int^{(j)}}$ has dimension $1$ and its kernel has codimension $1$. Thus, there is a rational $1$-form $\theta_j\in \Omega^1_M$ such that $\langle \theta_j \rangle^\perp = \ker({\rm int}^{(j)})$, and an isomorphism $\phi_j \colon \Gamma^{\rm rat}(\mathfrak a_1) \to \Gamma^{\rm rat}(\mathcal P)$ such that ${\rm int}^{(j)}(\vec X)(a) = \theta_j(\vec X)\phi_j(a)$. This yields:
$\Gamma^{\rm rat}(\mathfrak a_j) = \langle \theta_j \rangle\tens_{\mathbb C(M)} \Gamma^{\rm rat}(\mathcal P)$ and therefore:
$$\Gamma^{\rm rat}(\mathfrak a) = \langle \theta_1,\ldots,\theta_r\rangle \otimes \Gamma^{\rm rat}(\mathcal P).$$
As we explained before, the forms $\theta_j$ are defined on a finite covering of a Zariski open subset of $M$. However
we have $\Gamma^{\rm rat}(\ker(\mathfrak a))^{\perp} = \langle \theta_1,\ldots,\theta_r\rangle$. Now, $\mathcal A = \ker(\mathfrak a)$
is a rational linear sub-bundle of ${\rm TM}$ and we have $\mathfrak a = \mathcal A^\perp \otimes \mathcal P$ as stated. 
\end{proof}

\subsection{Construction of the invariant foliation} Motivated by Lemma \ref{lm:morimoto} we define
$\mathcal A = \ker(\mathfrak a)$ and $\mathcal H = \ker(\mathfrak g^0)$ that we see as rational linear bundles $\mathcal H \subset\mathcal A \subset {\rm T}M$.

\begin{lemma}
$ \mathcal H \cap \mathcal P = \{0\}$. 
\end{lemma}
\begin{proof} 
Take $v\in \mathcal H \cap \mathcal P$. We have $\mathfrak g^0v = \mathfrak{sl}(\mathcal P)v = 0$ and then $v = 0$.
\end{proof}

In order to prove the Frobenius integrability of the linear bundle $\mathcal H$ we have to discuss the symbol of differential operators. The $\mathcal D_M$-module $N$ is graded by the order,
$$N_0 = \rho^*\Omega^1_S \subset N_1 \subset N_2 \subset \ldots \subset N$$
Given a differential operator in $N_k$ its class in $\bar{N}_k = N_k/N_{k-1}$ is called its symbol. We may see, equivalently, that the symbol of a differential operator is its homogeneous part of higher order. 

An application of the graded bundle $\mathfrak g$ is that the space of symbols of order $k$ in $N$ is encoded in the bundle $\mathfrak g^{k-1}$. We only need to discuss symbols of order $1$. The coupling of a first order differential operator, 
$$\Theta = \sum \vec X_j \otimes \ell_j \in \mathfrak X_M\tens_{\mathbb C(M)} \Omega_M^1$$
with the linear part $g^1$ of a vector field $\vec g$ vanishing at $p\in M$ is given by:
$$\Theta(g^1) = \Theta(\vec g)(p) = \sum \vec X_j(\ell_j(\vec g)) = 
\sum \left({\rm Lie}_{\vec X_j}\ell_j(\vec g_p) + \ell[\vec X,\vec g]_p)\right) =
- \ell(g^1(\vec X_p))$$

\begin{proposition}
The rational linear bundle $\mathcal H = \ker(\mathfrak g^0)$ is a singular foliation in $M$.
\end{proposition}
\begin{proof} 
From the above discussion if follows that a rational vector field $\vec X$ is tangent to $\mathcal H = {\rm ker}(\mathfrak g^0)$ if and only if for any rational $1$-form $\ell$ in $M$ the symbol $X\otimes \ell$ is in $\bar N_1$.

Let us consider $\vec X$ and $\vec Y$ two rational vector fields tangent to $\mathcal H$ and consider any $\ell \in \Gamma^1_M$. There are $\alpha, \beta \in \Omega^1_M$ (differential operators of order zero) such that $D_1 = \vec X\otimes \ell + \alpha$ and $D_2 = \vec Y \otimes \ell + \beta$ are in $N_1$. Moreover, by the same reason, there are $\alpha', \beta' \in \Omega^1_M$ (differential operators of order zero) such that $D_3 \vec X\otimes \beta + \beta'$ and $D_4 =\vec Y \otimes \alpha + \alpha'$ are in $N_1$. Then,
$$\vec X\circ D_2 - \vec Y \circ D_1 - D_3 + D_4 = [\vec X,\vec Y]\otimes \ell + (\alpha' - \beta')$$
is in $N_1$. We have seen that for all $\ell\in \Omega^1_M$ the symbol
$[\vec X,\vec Y]\otimes \ell$ is in $\bar N_1$ and from this it follows that $[\vec X,\vec Y]$ is tangent to $\mathcal H$.
\end{proof}

\begin{lemma} \label{mathcal G is rho projectable}
$ \mathcal H$ is $\rho$--projectable.
\end{lemma}
\begin{proof} 
Let $\vec X$ be a vector field inside $\mathcal L$ that do not vanish at $p$ and is  tangent to the fibers of $\rho$. Let $\sigma_t$ be the flow of $\vec X$ at time $t$.
As $\vec X$ is in $\mathcal L$, the flows of $X$ transform $\mathcal L$ into $\mathcal L$. Then $d\sigma : {\rm T}_p M \to {\rm T}_{\sigma(p)}M$ conjugates $\mathcal H_p$ with $\mathcal H_{\sigma(p)}$. We get an isomorphism $\mathcal H_p \cong \mathcal H_{\sigma(p)}$. As $\rho\circ \sigma^{-1} = \rho$ it follows $d_p\rho(\mathcal H_p) = d_{\sigma(p)}\rho( \mathcal H_{\sigma(p)})$.
Then $d_p\rho(\mathcal G_p) = d_{q}\rho( \mathcal G_{q})$ whenever $p$ and $q$ can be joined by an integral curve in $\mathcal L$. As the action of $\mathcal L$ is transitive in the fibers of $\rho$ the lemma follows.
\end{proof}

From now on, let $\mathcal F$ be the projection $\rho_*\mathcal H$. Note that $\mathcal F$ is a singular foliation on $S$.

\begin{lemma}
We have the decomposition $\mathcal A = \mathcal P \oplus \mathcal H$ 
\end{lemma}

\begin{proof}
Consider the exact sequence
$$
\xymatrix{ 
0 \ar[r] & \mathfrak a \ar[r] & \mathfrak g^0 \ar[r] & \mathfrak{sl}(\mathcal P) \ar[r] & 0
}
$$
and take restriction to $\mathcal A$. 
$$
\xymatrix{ 
0 \ar[r] & \mathfrak g^0|_{\mathcal A} \ar[r] & \mathfrak{sl}(\mathcal P) \ar[r] & 0
}
$$
Then $ \mathfrak g^0|_{\mathcal A} \subset \mathrm{End}(\mathcal A)$ is a Lie algebra bundle isomorphic to $ \mathfrak{sl}(\mathcal P)$, so it is a bundle with simple Lie algebra fibers.

We have that $\mathfrak g^0|_{\mathcal A}$ preserves $\mathcal P$. By the simplicity of $\mathfrak g^0|_{ \mathcal A} $ the invariant bundle $\mathcal P$ admits an invariant suplementary. Thus, there exist a rational $\mathfrak g^0$-invariant sub-bundle $\mathcal K \subset A$ such that $\mathcal A = \mathcal K \oplus \mathcal P$. On the other hand the image of $\mathfrak g^0|_{\mathcal A}$ is $\mathcal P$ so it follows that $\mathfrak g^0|_{\mathcal K} = 0$ and then $\mathcal K = \mathcal H$.
\end{proof}

\begin{remark}
The previous lemma has important direct consequences:
\begin{enumerate}
    \item $\mathcal A$ is integrable and $\rho$-projectable.
    \item $\rho_*\mathcal A = \rho_* \mathcal H = \mathcal F$.
    \item ${\rm rank}\,\mathcal H = {\rm rank}\,\mathcal F$. 
    \item $\mathcal H$ is a partial $\mathcal F$-connection (in the sense of \cite{malgrange2010pseudogroupes}).
\end{enumerate}
\end{remark}


\subsection{Final step of the proof} After defining $\mathcal H$ as the kernel of linear isotropy let us consider $\overline{\mathcal L}$ the $\mathcal D$-Lie algebra defined by equation \eqref{eq:km} in the conclusion of Theorem \ref{th:kiso-mori}.
$$
\overline{\mathcal L}  = \Big\{ \vec X\in J({\rm T}M/M) \mid d\rho(\vec X) =0, \,\, \mathrm{Lie}_{\vec X} \omega \equiv 0 \,\, \mathrm{mod} \,\, \rho^* \Omega^1_S,$$ 
\vspace{-4mm}
$$
 \quad\quad \quad\forall \vec Y\in \mathfrak X_{\mathcal H} \,\, [\vec X,\vec Y]\in \mathfrak X_{\mathcal H}
\Big\}.
$$
To complete the proof of Theorem \ref{th:kiso-mori} we have to prove $\mathcal L = \overline{\mathcal L}$. The inclusion $ \mathcal L \subseteq \overline{\mathcal L}$ is elementary, we already know that $\omega|_{\mathcal P}$ and   $\mathcal H$ are $\mathcal L$-invariants. Then, our objective is to prove $\overline{\mathcal L}\subseteq \mathcal L$. Recalling Definition \ref{df:gk}, we construct the symbols of $\mathcal L$ and $\overline{\mathcal L}$.
$$\overline{\mathfrak g}^k = \overline{\mathcal L}^{\geq k+1}/\overline{\mathcal L}^{\geq k+2},\quad
\overline{\mathfrak g} = \bigoplus_{k=0}^{\infty} \overline{\mathfrak g}^k.$$
We have $\mathfrak g^k \subseteq \overline{\mathfrak g}^k$ as a linear bundle and $\mathfrak g \subseteq \overline{\mathfrak g}$ as a linear graded bundle. Moreover, the homogeneous part of smaller degree of a Lie bracket depends only of the homogeneous part of smaller degree of the factors. Therefore the Lie bracket is defined as a graded operation,
$$
[\,\,\,,\,\,]\colon 
\overline{\mathfrak g}^k \times_M \overline{\mathfrak g}^\ell \to \overline{\mathfrak g}^{k+\ell}.
$$
Thus $\overline{\mathfrak g}$ is bundle by graded Lie algebras, and $\mathfrak g$ is a sub-bundle by graded Lie subalgebras of the former. 
Let us consider $\mathcal L = \{\mathcal L_k\}$ and 
$\overline{\mathcal L} = \{\overline{\mathcal L}_k\}$ as projective systems of finite rank vector bundles. Note that, by definition of $k$-jet we have:
$$\mathcal L_k = \mathcal L/ \mathcal L^{\geq k+1}, \quad 
\overline{\mathcal L}_k = \overline{\mathcal L}/ \overline{\mathcal L}^{\geq k+1}$$
We also have $\mathcal L_0 = \overline{\mathcal L}_0 = \mathcal P$ which is of rank $m$. It follows that for $k>0$
$${\rm rank}( \mathcal L_k ) = m + \sum_{j=0}^{k-1} {\rm rank}(\mathfrak g^k), \quad
\dim( \overline{\mathcal L}_k ) = m + \sum_{j=0}^{k-1} {\rm rank}(\overline{\mathfrak g}^k)$$
Thus, $\mathfrak g^k = \overline{\mathfrak g}^k$ for all $k$ if and only if 
$\mathcal L_k = \overline{\mathcal L}_k$ for all $k$ if and only if 
$\mathcal L_p = \overline{\mathcal L}_p$. Therefore, \emph{the proof of Kiso-Morimoto theorem is reduced to the proof of the equality of the symbols of $\mathcal L$ and $\overline{\mathcal L}$.}\\

{\bf Notation:} It suffices to prove $\overline{\mathfrak g}_p =  \mathfrak g_p$ for a regular point $p$ of the symbols. Then we fix such a point $p$; we assume that $p$ is not a singular point of $\mathcal H$ or $\omega|_{M_{\rho(p)}}$. From now, in order to simplify the writting, we will omit the subindex $p$ of the symbol, writting $\mathfrak g^k$ instead of $\mathfrak g^k_p$ and so on. We also fix the notation $\mathfrak h = \mathfrak{sl}(\mathcal P_p)\subset \mathfrak g^0$.\\

We consider around $p\in M$ a system of adapted analytic coordinates $s_1,$ $\ldots,$ $s_q,$ $t_1,$ $\ldots,$ $t_r,$ $x_1,\ldots, x_m$, vanishing at $p$ with the following properties. 
\begin{itemize}
    \item[(a)] $s_1,\ldots,s_q,t_1,\ldots,t_r$ form system of coordinates in $S$ around $\rho(p)$, that is, as local functions in $M$ they are  constant along the fibers of $\rho$.
    \item[(b)] The foliation $\mathcal F$ in $S$ rectified by the coordinates $t,s$, that is:
    $$
    \mathcal F = \big\{ds_1 = \ldots =ds_q = 0 \big\} = \left\langle \frac{\partial}{\partial t_{1}}, \ldots,\frac{\partial}{\partial t_{r}}\right\rangle. 
    $$
    \item[(c)] The functions $x_i$ are first integrals of $\mathcal H$, thus on $M$: 
    $$\mathcal H = \left\langle \frac{\partial}{\partial t_{1}}, \ldots,\frac{\partial}{\partial t_{r}}\right\rangle. $$
    \item[(d)] The volume form $\omega$ is written in canonical form with respect to the $x_i$ functions,
    $$
    \omega =  dx_1\wedge \cdots \wedge dx_m.
    $$
\end{itemize}

Let us consider $\vec X\in \overline{\mathcal L}_p$ be a formal vector field.
its local expression in such system of adapted coordinates has the form 
$$
\vec X= \sum_j f_j(x,s,t) \frac{\partial}{\partial x_j}.
$$
where $f_j\in \mathbb C[[x,s,t]]$. The vector field $X$ preserves $\mathcal H$ therefore for $j=1,\ldots,r$ we have that $\mathrm{Lie}_X \frac{\partial}{\partial t_{j}}$ is a vector field tangent to $\mathcal H$:
$$
\mathrm{Lie}_{\vec X} \frac{\partial}{\partial t_{j}} = -\Big[ \frac{\partial}{\partial t_{j}}, \vec X\Big]
= -\sum_{i=1}^m \frac{\partial f_i}{\partial t_{j}} \frac{\partial}{\partial x_i}
$$
Therefore, $\frac{\partial f_i}{\partial t_{j}} =0$ for all $j=1,\ldots,r$. We obtain that the expression of $X$ has the form,
\begin{equation}\label{localLbar}
\vec X = \sum_{i=1}^m f_i(x,s) \frac{\partial}{\partial x_i}.
\end{equation}
With $f_j(x,s)\in \mathbb C[[x,s]]$. We also know that $X$ preserves $\omega$. Finally we have that $X$ is in $\overline{\mathcal L}$ if and only if its local expression is of the form \eqref{localLbar}, satisfying
$$
\mathrm{div}_\omega X = \sum_{i=1}^m \frac{\partial f_i}{\partial x_i} = 0.
$$

Let $\mathbb C[x,s]_{k+1}$ the space of homogeneous polynomials of degree $(k+1)$ in the variables $x_1,\ldots,x_m,s_1,\ldots,s_q$. From the local expression of $X$ in $\overline{\mathcal L}$ around $p$ we have that the vector space $\overline{\mathfrak g}^k$ is identified with a subspace of $\bigoplus_{i=1}^m \mathbb C[x,s]_{k+1} \frac{\partial}{\partial x_i}$.\\

Let $P_{(k+1)}(x,s)$ be an homogeneous polynomial of degree $(k+1)$. It can be decomposed in a unique way as a sum,
$$
P_{k+1}(x,s) = P_{k+1}^{(0)}(x) + P_{k+1}^{(1)}(x,s) + \ldots  + P_{k+1}^{(k)}(x,s) + P_{k+1}^{(k+1)}(s),
$$
where $P_{k+1}^{(j)}$ is in $(\mathbb C[x]_{k+1-j})[s]_j$, this is, it is an homogeneous polynomial of degree $j$ in the variables $s_1,\ldots,s_q$ whose coefficients are homogeneous polynomials in $x_1,\ldots,x_m$ of degree $(k+1-j)$.\\

Given $X\in \overline{\mathfrak g}^k$ we can define the valuation $\mathrm{val}_s(X)$ as the minimum grade in $s$ appearing in $X$. It allows us to define a filtration in $\overline{\mathfrak g}^k$.

\begin{definition}
Let $( \overline{\mathfrak g^k})^{ \geq \ell}$ be the set of elements of $ \mathfrak g^k$ with valuations in $s$ bigger than $\ell$,
$$ 
(\overline{\mathfrak g}^k)^{\geq \ell} = \{ X\in \overline{\mathfrak g}^k \mid \mathrm{val}_s(X) \geq \ell \}.
$$
so that we have a filtration,
$$
(\overline{\mathfrak g}^k)^{\geq k+2} = \{0\} 
\subset (\overline{\mathfrak g}^k)^{\geq k+1} 
\subset 
\ldots
\subset 
(\overline{\mathfrak g}^k)^{\geq 0} = \overline{\mathfrak g}^k.
$$
Let $\overline{\mathfrak g}^{k,\ell}$ be the
consecutive quotient $(\overline{\mathfrak g}^{k})^{ \geq \ell}/(\overline{\mathfrak g}^{k})^{\geq \ell+1}$. 
\end{definition}

By restriction of this filtration to the subspace
$\mathfrak g^k$ we define the consecutive quotients $\mathfrak g^{k,\ell}$. By construction $\mathfrak g^{k,\ell} \subseteq \overline{\mathfrak g}^{k,\ell}$. Since we have a finite filtration, we have that $\mathfrak g^{k} = \overline{\mathfrak g}^{k}$ if and only if $\mathfrak g^{k,\ell} = \overline{\mathfrak g}^{k,\ell}$ for all $\ell=0,\ldots,k+1$. Thus, \emph{the proof of Kiso-Morimoto theorem is reduced to the comparison of the spaces $ \mathfrak g^{k,\ell}$ to $ \overline{ \mathfrak g}^{k,\ell}$.}\\

From the fundamental hypothesis on $\mathcal L$ we have the following.

\begin{lemma}\label{g01=OL{g01}}
We have $ \mathfrak g^{0,1} = \overline{ \mathfrak g}^{0,1} = \mathfrak a$.
\end{lemma}

\begin{proof}
For $v^1\in \mathfrak g^{0,1}\subset \mathfrak g^0$ then its restriction to $s=0$ is 0. This means that $v^1\in \mathfrak a$.
\end{proof}

Let us denote by $\mathfrak h^k \subset \bigoplus_{i=1}^m \mathbb C[x]_{k+1} \frac{\partial}{\partial x_i}$ the Lie algebra of divergence free vector fields not depending of the variables $s$. Note that $\mathfrak h^0 = \mathfrak{sl}(\mathcal P_p)\subset \mathfrak g^0$.

\begin{lemma}\label{gk0= OL{gkl}}
$\mathfrak g^{k,0} = \overline{ \mathfrak g}^{k,0} \cong \mathfrak h^k$. 
\end{lemma} 

\begin{proof}
By (Hyp2) in Theorem \ref{th:kiso-mori} we have $\mathfrak h^k\subset\mathfrak g^k$ for all $k$. Taking the quotient is equivalent to take $s=0$. If follows easily that $\mathfrak h^k$ is a suplementary of $(\mathfrak g^k)^{\geq 1}$ in $\mathfrak g^k$ and of $(\overline{\mathfrak g}^k)^{\geq 1}$ in $\overline{\mathfrak g}^k$.
\end{proof}

Next step is to propagate the equalities by means of the Lie bracket. 

\begin{lemma}
The Lie bracket in $\overline{\mathfrak g}$ is compatible 
with the filtration of the spaces $\overline{\mathfrak g}^k$, and therefore it induces a Lie bracket between the intermediate quotients,
$$
[\,\,\,,\,\,]\colon 
\overline{\mathfrak g}^{k,\ell} \times \overline{\mathfrak g}^{k',\ell'} \to \overline{\mathfrak g}^{k+k',\ell+\ell'}.
$$
\end{lemma}

\begin{proof}
Let's take $X\in \overline{\mathfrak g}^{k,\ell}$ and $Y\in \overline{\mathfrak g}^{k',\ell'}$. We write here $ \partial_i$ for $\partial/\partial x_i$.  We have the Taylor development $X = \sum P_i \partial_i + R_i \partial_i$ where $P_i$ is an homogeneous polynomial in $x$ and $s$, with grades $k+1-\ell$ and $\ell$ respectively. The $R_i$ have degree $\ell+1$ in $s$ and $k+2$ in $x$. Similarly write $Y = \sum P'_i \partial_i + R'_i \partial_i$ where $P'_i$ is an homogeneous polynomial in $x$ and $s$, with grades $k'+1'-\ell$ and $\ell'$ respectively. The $R'_i$ have degree $\ell'+1$ in $s$ and $k'+2$ in $x$. The bracket can be expressed as
$$
[X,Y] =\sum_{i,j}\big(P_i \partial_iq_j - Q_i \partial_iP_j\big) \partial_i + \big( R_i \partial_i Q_j - R_i' \partial_i P_j + (P_i+R_i) \partial_i R_j' - (Q_i + R_i') \partial_iR_j \big) \partial_i
$$
Observe that:
\begin{enumerate}
    \item $P_i \partial_iq_j - Q_i \partial_iP_j$ \quad have degree $\geq \ell+\ell'$ in $s$ and $k+k'+1$ in $x$,
    \item $R_i \partial_i Q_j - R_i' \partial_i R_j$ \quad have degree $\geq \ell+\ell'+1$  in $s$,
    \item $(P_i+R_i) \partial_i R_j' - (Q_i + R_i') \partial_iR_j$ \quad have degree $\geq \ell+\ell'+1$  in $s$.
\end{enumerate}
It follows that the bracket is an element in $\overline{\mathfrak g}^{k+k', \ell+\ell'}$.
\end{proof}

$$
\xymatrix{ 
\vdots & & & \vdots & & \vdots &  &
\\
\mathfrak g^2 & \overline{\mathfrak g}^{2,0}
=\mathfrak g^{2,0} 
\ar@/^{5mm}/[rr]^{[\mathfrak a,\,\,\,]}
& \oplus & 
\mathfrak g^{2,1}  \ar@/^{5mm}/[rr]^{[\mathfrak a,\,\,\,]}
& \oplus &
\mathfrak g^{2,2} \ar@/^{5mm}/[rr]^{[\mathfrak a,\,\,\,]}
& \oplus & \mathfrak g^{2,3}
\\
\mathfrak g^1 & \overline{\mathfrak g}^{1,0} =\mathfrak g^{1,0} \ar@/^{5mm}/[rr]^{[\mathfrak a,\,\,\,]}
& \oplus & 
\mathfrak g^{1,1}  \ar@/^{5mm}/[rr]^{[\mathfrak a,\,\,\,]}
& \oplus &
\mathfrak g^{1,2} 
\\
\mathfrak g^0 & \overline{\mathfrak g}^{0,0} =\mathfrak g^{0,0} 
 \ar@/^{5mm}/[rr]^{[\mathfrak a,\,\,\,]}
& \oplus & 
\mathfrak g^{0,1}  = \mathfrak a
}
$$

Observe that the first column of the bi-graded diagram we have equalities. Assume that the Lie bracket,
$$
[\,\,\, ,\,\,]\colon \mathfrak a\times \overline{\mathfrak g}^{k,\ell}\to \overline{\mathfrak g}^{k,\ell+1}
$$
is surjective for all $k$ and $\ell\leq k$. Then, by finite induction on $\ell$ for all $k$, following the horizontal lines of the above diagram we have
${\mathfrak g}^{k,\ell} = \overline{\mathfrak g}^{k,\ell}$ for all $k$ and $\ell$. 
In order to proof the surjectiveness of the Lie bracket we need the following technical result.

\begin{lemma}\label{lm:lemma_de_Guy}
For any $X\in \mathfrak h^k$ there is $Y\in \mathfrak h^{k+1}$ such that $[\frac{\partial}{\partial x_1}, Y ]= X$. 
\end{lemma}

\begin{proof}
We have $X = \sum P_j(x)\frac{\partial}{\partial x_j}$ where the polnomials $P_j(x)$ are homogeneous of degree $k+1$. Let us consider the homogeneous polynomials $Q_j$ verifying,
$$\frac{\partial Q_i}{\partial x_1}(x) = P_j(x).$$
Let us set,
$$Y_0 = \sum_{j} Q_j(x)\frac{\partial}{\partial x_j}, \quad
E = \sum_{j\geq 2} x_j\frac{\partial}{\partial x_j},
\quad R(x) = {\rm div}_\omega(Y_0).$$
Note that
$$\frac{\partial}{\partial x_1} R(x) =
\sum_j \frac{\partial Q_j}{\partial x_j\partial x_1} =
\sum_j \frac{\partial P_j}{\partial x_j} = {\rm div}_{\omega}(X) = 0,$$
and therefore $R(x)$ is an homogeneous polynomial in $x_2,\ldots,x_m$ of degree $k+1$. In particular we have 
$$E\,R(x) = (k+1)R(x), \quad {\rm div}_{\omega}(R(x)E) = (k+m)R(x).$$
We conveniently take $$Y = Y_0 - \frac{1}{k+m} R(x) E$$ 
which satisties,
$$\left[\frac{\partial}{\partial x_1}, Y\right] = X, \quad
{\rm div}_{\omega}(Y) = R(x) - \frac{k+m}{k+m} R(x) = 0.$$
and thus $Y$ is the required element of $\mathfrak h^{k+1}$.
\end{proof}

Finally, by the above argument, the proof of Kiso-Morimoto theorem is reduced to the following fundamental Lemma.

\begin{lemma}\label{BracketSurjectiveLemma}
 Lie bracket,
$$
[\,\,\, ,\,\,]\colon \mathfrak a\times \overline{\mathfrak g}^{k,\ell}\to \overline{\mathfrak g}^{k,\ell+1}
$$
is surjective for all $k$ and $\ell\leq k$.
\end{lemma}

\begin{proof}
We proceed by fixing $k$ and doing induction on $\ell$. The initial steps are given by lemmas \ref{g01=OL{g01}} and \ref{gk0= OL{gkl}}. Assume that the equality holds for $\ell$ and show for $\ell+1$ \emph{i.e.} let us see that  $\mathfrak g^{k,\ell+1} =\overline{\mathfrak g}^{k,\ell+1}$. Let $Q\in \mathrm{sym}^1 (s)$, $P\in \mathrm{sym}^\ell(s)$ and $X\in \mathfrak h^{k-l}$ such that $QPX\in \overline{\mathfrak g}^{k,\ell+1}$. By the lemma \ref{lm:lemma_de_Guy}, there exists $Y\in \mathfrak h^{k-\ell+1}$ such that $[\frac{\partial}{\partial x_1}, Y] = X$.  We have $Q\otimes \frac{\partial}{\partial x_1} \in \mathfrak g^{0,1}$ and $P\otimes Y \in \mathfrak g^{k,\ell}$. Therefore we get the following expression:
$$
[Q \otimes \frac{\partial}{\partial x_1}, P\otimes Y] = QPX
$$
because $Y(Q)= \frac{\partial}{\partial x_1}(P) =0$. As the elements of the form $QPX$ form a basis for $ \overline{\mathfrak g}^{k,\ell+1}$ we get the lemma. 
\end{proof}

Finally we have $\mathfrak g = \overline{\mathfrak g}$ and therefore $\mathcal L = \overline{\mathcal L}$; this finishes the proof of Kiso-Morimoto Theorem \ref{th:kiso-mori}.


\section{On the definition of $\mathcal D$-groupoid}\label{App:A}

In this appendix we show the equivalence between Malgrange's definition of $\mathcal D$-groupoid (Definition 5.2 in \cite{malgrange2010pseudogroupes}) and ours (Definition \ref{def:D_groupoid}). A complete reference with different proofs is D. Davy's PhD thesis \cite{Damien2016specialisation}. In order to revisit the definition we need to examine the differential structure of the jet bundle.\\

First, let us consider $P\to M$ a bundle with $P$ and $M$ affine and smooth varieties.
Whenever we need we replace $M$ by a suitable Zariski dense open subset. Therefore,
we may assume that we have functionally independent functions $x_1,\ldots, x_m$ such that the map $M\to \mathbb C^m$ is an open cover of an affine subset of $\mathbb C^m$ and thus the functions $x_1,\ldots, x_m$ form a local system of coordinates in a neighbourhood (in the usual topology) of any point. Let us set the notation $\vec D_i = \frac{\partial}{\partial x_i}$ for partial derivative operators.\\

For each order $k$ we consider the $k$-jet bundle of sections $J_k= J_k(P/M)$. We have submersions,
$$J = \lim_{\leftarrow} J_\ell  \to \ldots \to J_k \to \ldots \to J_1 \to P \to M$$
and a local mechanism of $k$-jet prolongation of analytic sections,
$$j^k\colon \Gamma(U,P) \to \Gamma(U,J_k),\quad (j^k u)(p) = j^k_pu.$$
The rings of regular functions (that give algebraic structure to the sets $J_k$ ) are defined inductively in the following way. A regular vector field $\vec X$ in $M$ extends to a total derivative operator,
$$\vec X^{\rm tot}\colon \mathcal O_{J_k}\to \mathcal O_{J_{k+1}}$$
where $(\vec X^{\rm tot}f)(j^{k+1}_pu) = \vec X_p((j^ku)^*f)$, starting at $\mathcal O_{J_0} = \mathcal O_P$. Then we take:
$$\mathcal O_{J_1} = \mathcal O_P[\vec D_j^{\rm tot}\mathcal O_P: j = 1,\dots m]$$
$$\mathcal O_{J_2} = \mathcal O_P[\vec D_j^{\rm tot}\mathcal O_{J_1}: j = 1,\dots m]$$
$$\ldots$$
$$\mathcal O_{J_k} = \mathcal O_P[\vec D_j^{\rm tot}\mathcal O_{J_{k-1}}: j = 1,\dots m]$$
Then we have that 
$$\mathcal O_J = \bigcup_{k=0}^\infty \mathcal O_{J_k}$$
is a $\vec D$-ring with $\vec D = \{\vec D_1^{\rm tot},\ldots,\vec D_m^{\rm tot}\}$. This realizes the bundles $J_k$ as affine schemes. A system of PDE's is, by definition, the set of zeroes $Z \subset J$ of a radical $\vec D$-ideal $\mathcal I \subset \mathcal O_{J}$. The zero set $Z_k\subset J_k$ of intersection $\mathcal I_r = \mathcal I \cap \mathcal O_{J_r}$ consists of the equations of $Z$ of order $\leq r$. We identify $Z$ with the sequence $Z = \{Z_k\}_{k\in \mathbb N}$. The following are, not so elementary, but well known facts.
\begin{itemize}
    \item[(a)] There is a minimum $r$ such that $\mathcal I_r$ spans $\mathcal I$ as a radical $\vec D$-ideal, this $r$ is the so called order of $Z$ (Ritt-Raudenbush)
    \item[(b)] If $\mathcal I$ is not trivial and contains no equations of order $0$ then there is a Zariski open subset $U$ such that for $k>1$ the projections $Z_k|_U\to Z_{k-1}|_U$ are submersions, and futhermore, there is a $k_0\in \mathbb N$ such that for $k>k_0$ they are affine bundles (Generic involutivity).
    \item[(c)] $Z$ is characterized by its local analytic solutions. By a local analytic solution we mean a local analytic section $u\colon U\to P$ such that for all $p\in U$, $j_pu\in Z$ (Differential Nullstellensatz).
\end{itemize}

Now we consider $P = M\times M  \to M$ with the projection in the first factor. For each $k$ define $J^*_k$ the set of $k$-jets of graphs of local biholomorphisms. This is an affine open subset of $J_k$ complementary of the zero locus of the Jacobian (which is an hypersurface defined by an equation of order $1$). We have, again submersion and ring inclusions,
$$J^* = \lim_{\leftarrow} J^*_\ell  \to \ldots \to J^*_k \to \ldots \to J_1 \to P \to M$$
$$\mathcal O_M \subset \mathcal O_P \subset \mathcal O_{J_1^*}\subset \ldots \subset \,\,\mathcal O_{J^*} = \bigcup_{k} \mathcal O_{J_k}.$$

Since the elements of $J^*$ are formal biholomorphisms it follows that $J^*\to M\times M$ is a groupoid, an alternative construction for ${\rm Aut}(M)$. Now we can give a definition of $\mathcal D$-groupoid which is clearly equivalent to Definition 5.2 in \cite{malgrange2010pseudogroupes}.

\begin{definition}\label{df:D_grupoidM}
We say that a closed subset $\{\mathcal G_k\}_{k\in\mathbb N} = \mathcal G\subset J^*$ is a $\mathcal D$-groupoid if:
\begin{enumerate}
\item[a)] $\mathcal G$ is the zero set of a radical $\vec D$-ideal $\mathcal I \subset\mathcal O_{J^*}$.
\item[b)] $\mathcal G_k$ contains the identity section of $J^*_k$ and it is stable by inversion.
\item[c)] There is a dense Zariski open subset $U\subset M$ such that $\mathcal G_k|_{U\times U}$ is a Lie subgroupoid of $J_k^*(U\times U/U) \cong {\rm Aut}_k(U)$ for all $U$.
\end{enumerate}
\end{definition}

\begin{lemma}
Any $\mathcal D$-groupoid in the sense of Definition \ref{df:D_grupoidM} is a $\mathcal D$-groupoid in the sense of Definition \ref{def:D_groupoid}. 
\end{lemma}

\begin{proof}
Let us consider $\mathcal G = \{\mathcal G_k\}$ a $\mathcal D$-groupoid in the sense of Definition \ref{df:D_grupoidM}. By Proposition \ref{pro:Galois_correspondence1} $\mathcal G_k$ is completely determined by its field of invariant ${\rm Inv}(\mathcal G_k)\subset\mathbb C({\rm R}M)$. Thus, $\mathcal G$ is completely determined by its field of invariants,
$$\mathbb F = \bigcup_k {\rm Inv}(\mathcal G_k)$$
which is by construction a $\Gamma$-invariant field containing $\mathbb C$. We only need to check that {\color{red}it is a} $\Delta$-field. Let us denote ${\rm Sol}(\mathcal G)$ its pseudogroup of solutions, that is, local biholomorphisms $\sigma\colon U\xrightarrow{\sim} V$ between open subsets (in the usual topology) of $M$ such that for all $p\in U$ $j_p\sigma\in\mathcal G$. Sigma, naturally induces a $\Delta$-ring morphism,
$$\sigma^*\colon \mathcal O_{{\rm R}V}\to \mathcal O_{{\rm R}U}$$
that extends to the corresponding field of fractions, containing $\mathbb C({\rm R}M)$. A rational differential function is $\sigma$-invariant if $f|_{{\rm R}U} = \sigma^*(f|_{{\rm R}V})$, which is equivalent to say that is is invariant by $j_p\sigma$ for all $p\in U$. By Ritt's differential Nullstellensazt we have that the jets of local transformations in ${\rm Sol}(\mathcal G)$ are dense in $\mathcal G$. Therefore, a rational differential function is $\mathcal G$-invariant if and only if it is ${\rm Sol}(\mathcal G)$-invariant. Now, if $f$ is $\sigma$-invariant then for all $j = 1,\ldots, m$,
$$\delta_j(f|_{{\rm R}U}) = \delta_j(\sigma^*(f|_{{\rm R}V})) = \sigma^*(\delta_jf|_{{\rm R}V}),$$
and therefore $\delta_jf$ is a also $\sigma$-invariant.It follows that $\mathbb F$ is a $\Delta$-field and we have
$$\mathcal G = {\rm Sym}_\Delta(\mathbb F)$$
which is a $\Delta$-groupoid in the sense of Definition \ref{def:D_groupoid}.
\end{proof}

In order to prove the other way of the equivalence we need to examine how to compute the ideal of a Lie groupoid of gauge transformations.

\begin{lemma}\label{lm:lmA}
Let $P\to M$ be a $G$-principal bundle, $\varphi\colon M\to P$ a regular section, and $\mathcal G$ a Lie groupoid of gauge transformations of $P$. Let us assume that ${\rm Inv}(\mathcal G) = \mathbb C(I_1,\ldots,I_\ell)$ such that $\varphi^*(I_k)$ is a regular function in $M$ for $k = 1,\ldots,\ell$. Let us consider,
$\varphi_{\rm iso}\colon $ ${\rm Iso}(P)\to P$, $\sigma\mapsto \sigma\circ \varphi_{s(\sigma)}$. Then:
\begin{itemize}
    \item[(a)] $\varphi_{\rm iso}^*(I_k)$ is regular in an affine neighborhood $U\subset {\rm Iso}(P)$ of $\mathcal G$ for all $k = 1,\ldots,\ell.$
    \item[(b)] $\mathcal G$ is the zero set of the ideal 
    $(\varphi^*_{\rm iso}(I_1)-\varphi^*(I_1),\ldots,\varphi^*_{\rm iso}(I_\ell)-\varphi^*(I_\ell))\subset \mathcal O_U$.
\end{itemize}
\end{lemma}

\begin{proof}
Let us consider the action $a\colon {\rm Iso}(P)\times_M P \to P$. By definition,a rational function $I\subset\mathbb C(P)$ is $\mathcal G$-invariant if and only if $a^*(I) = I$. Let $\mathcal W$ be a maximal affine open subset in which $I_1,\ldots,I_k$ are regular. Then $(\mathcal G\times_M P)\cap a^{-1}(\mathcal W)$ is the zero set of the ideal:
$$(a^*(I_1) - I_\ell, \ldots, a^*(I_\ell) - I_\ell).$$
For any $\sigma\in\mathcal G$ we have $(a^*I)(\sigma, \varphi_{s(\sigma}) = I(\varphi_{s(\sigma)})$ and therefore  the map 
$$({\rm Id}, \varphi_{\rm iso})\colon {\rm Iso}(P)\to \mathcal {\rm Iso}(P)\times P$$
maps $\mathcal G$ into $a^{-1}(\mathcal W)$. By taking $U = ({\rm Id}, \varphi_{\rm iso})^{-1}(a^{-1}(\mathcal W))$ and observing that\linebreak
$({\rm Id}, \varphi_{\rm iso})^*(a^*(I_k) - I_k)$ $=$ $\varphi^*_{\rm iso}(I_k) - \varphi^*(I_k)$ we finish the proof. 
\end{proof}

\begin{lemma}
Any $\mathcal D$-groupoid in the sense of Definition \ref{def:D_groupoid} is a $\mathcal D$-groupoid in the sense of Definition \ref{df:D_grupoidM}. 
\end{lemma}

\begin{proof}
Let $\mathcal G = \{\mathcal G_k\}$ a $\mathcal D$-grupoid in the sense of Definition \ref{def:D_groupoid}. Let $\mathcal I$ be the radical ideal of $\mathcal O_{J^*}$ of regular functions vanishing in $\mathcal G$ and $\mathcal I_k = \mathcal I\cap \mathcal O_{J^*_k}$. \\

The system of local coordinates $\vec x_1,\ldots, \vec x_m$ induces a global frame $\varphi\colon M \to {\rm R}M$ defined as $\varphi_p = j_0(\varepsilon \mapsto p + \varepsilon)$. The global frame $\varphi$ induces a regular submersion\footnote{This is the same as $\varphi_{\rm iso}$ in Lemma \ref{lm:lmA}.} $\varphi_J\colon J^*\to {\rm R}M$, $\sigma\mapsto \sigma\circ\varphi_{s(\sigma)}$. The global frame $\varphi$ is defined in such way that
$$\varphi^*(\delta_j I) = \vec D_j\varphi^*(I), \quad \varphi_J^*(\delta_j I) = \vec D_j^{\rm tot}\varphi^*_JI$$
for any rational differential function $I\in \mathbb C(M)$ whose domain intersects the graph of $\varphi$. If $I$ is indeed $\mathcal G$-invariant then $\varphi_J^*(I) - \varphi^*(I)$ vanishes identically along $\mathcal G$.\\

For any given $k$ let us consider $I_1,\ldots, I_{\ell}$ generators of ${\rm Inv}(\mathcal G_k)$. By means, if necessary, of replacing $M$ by an smaller open affine subset and some right translation by an element of $\Gamma_k$ in ${\rm R}_kM$ we may assume that we are under the hypothesis of Lemma \ref{lm:lmA}. Thus,
$$\mathcal I_k|_{U} = {\rm rad}(\varphi_J^*(I_1)-\varphi(I_1),\ldots, \varphi_J^*(I_\ell)-\varphi(I_\ell)).$$
Where $U$ is an affine neighborhood of $\mathcal G$. Let us consider $F\in \mathcal I_k$. By now, let us assume that $\mathcal G$ is irreducible. Then, there is $n$ such that $F^{n-1}$ is not identically $0$ on $\mathcal G$ and such that,
$$F^{n} = \sum_{j=1}^\ell G_j(\varphi_J^*(I_j)-\varphi^*(I_j)),$$
with $G_j\in \mathcal O_U\cap \mathbb C(J_k^*)$.
Then for any $\alpha = 1,\ldots,m$ we have,
$$nF^{n-1}\vec D^{\rm tot}_\alpha F =\sum_{j=1}^m
\left(
\vec D^{\rm tot}_\alpha G_j(\varphi_J^*(I_j)-\varphi^*(I_j)) + G_j(\varphi_J^*(\delta_{\alpha}I_j)-\varphi^*(\delta_{\alpha}I_j) \right),$$
therefore $nF^{n-1}\vec D^{\rm tot}_{\alpha}F$ vanishes along $\mathcal G$, so does $\vec D^{\rm tot}_\alpha F$. Let us see the non-irreducible case. Let $\mathcal G = \bigcup_{\beta} \mathcal G^{(\beta)}$ the decomposition of $\mathcal G$ in irreducible components. We take for each $\beta$ an affine open neighborhood $U^{(\beta)}$ of $\mathcal G^{(\beta)}$ contained in $U - \bigcup_{\gamma\neq\beta}\mathcal G^{(\gamma)}$. Then, for each $\beta$ there is $n_{\beta}$ such that $F^{n}\in \mathcal I_{k}|_{\mathcal U^{(\beta)}}$ and $F^{n-1}$ does not vanish identically on $\mathcal G^{(\beta)}$. The same argument proves that $\vec D^{\rm tot}_\alpha F$ vanishes along $\mathcal G^{(\beta)}$ for all $\beta$, and thus $\vec D^{\rm tot}_\alpha F\in \mathcal I$. It follows that $\mathcal I$ is a radical $\vec D$-ideal and $\mathcal G$ is a $\mathcal D$-groupoid in the sense of definition \ref{df:D_grupoidM}.
\end{proof}

From the above lemmas it follows that 
definitions \ref{def:D_groupoid} and \ref{df:D_grupoidM} are equivalent.

\subsection*{Acknowledgements}
We would like to thank Universit\'e de Rennes 1 and Universidad Nacional de Colombia for their hospitality and support. D. Bl\'azquez-Sanz has been partially funded by Colciencias project ``Estructuras lineales en geometr\'ia y topolog\'ia'' 776-2017 code  57708 (Hermes UN 38300).
G. Casale has been partially funded by  Math-AMSUD project ``Complex Geometry and Foliations''.  J. S. D\'iaz Arboleda has been partially funded by Colciencias program 647 ``Doctorados Nacionales''.  We want to thank the anonymous referee for his/her carefull reading, suggestions and corrections that greatly helped to improve the quality of the paper.

\bibliographystyle{plain}
\bibliography{references}

\begin{thebibliography}{10}

\bibitem{blazquez2019differential}
David Bl{\'a}zquez~Sanz, Guy Casale, and Juan~Sebasti{\'a}n D{\'\i}az~Arboleda.
\newblock Differential galois theory and isomonodromic deformations.
\newblock {\em SIGMA. Symmetry, Integrability and Geometry: Methods and
  Applications}, 15:055, 2019.

\bibitem{cantat2009dynamics}
Serge Cantat and Frank Loray.
\newblock Dynamics on character varieties and {Malgrange} irreducibility of
  painlev{\'e} vi equation.
\newblock In {\em Annales de l'institut Fourier}, volume~59, pages 2927--2978,
  2009.

\bibitem{casale2004groupoide}
Guy Casale.
\newblock {\em Sur le groupo{\"\i}de de {Galois} d'un feuilletage}.
\newblock PhD thesis, Universit{\'e} Paul Sabatier-Toulouse III, {\tt
  https://tel.archives-ouvertes.fr/tel-00012021}, 2004.

\bibitem{casale2009preuve}
Guy Casale.
\newblock Une preuve galoisienne de l'irr{\'e}ductibilit{\'e} au sens de
  {Nishioka-Umemura} de la 1{\`e}re {\'e}quation de {Painlev{\'e}}.
\newblock {\em Ast{\'e}risque}, 323:83--100, 2009.

\bibitem{casale2011}
Guy Casale.
\newblock An introduction to {M}algrange pseudogroup.
\newblock In L.~Di~Vizio and T.~Rivoal, editors, {\em Séminaires et
  Congr\`es}, volume~23, pages 89--113. Publication SMF, 2011.

\bibitem{casaledavy2020}
Guy Casale and Damien Davy.
\newblock Sp\'ecialisaton du groupoide de {G}alois d'un champ de vecteurs.
\newblock {\em arXiv\ math/2004.09122}, 2020.

\bibitem{cassidy2005galois}
Phyllis~J Cassidy and Michael~F Singer.
\newblock Galois theory of parameterized differential equations and linear
  differential algebraic groups.
\newblock In D.~Bertrand, B.~Enriquez, C.~Mitschi, C.~Sabbah, and R.~Schaefke,
  editors, {\em IRMA Lectures in Mathematics and Theoretical Physics},
  volume~9, pages 113--157. EMS Publishing house, 2005.

\bibitem{cassidy1989classification}
Phyllis~Joan Cassidy.
\newblock The classification of the semisimple differential algebraic groups
  and the linear semisimple differential algebraic {Lie} algebras.
\newblock {\em Journal of Algebra}, 121(1):169--238, 1989.

\bibitem{Damien2016specialisation}
Damien Davy.
\newblock {\em Sp\'ecialisation du pseudo-groupe de {M}algrange et
  irr\'eductibilit\'e}.
\newblock PhD thesis, Universit\'e Rennes 1 {\tt
  https://tel.archives-ouvertes.fr/tel-01491008}, 2016.

\bibitem{drach1898essai}
Jules Drach.
\newblock Essai sur une th{\'e}orie g{\'e}n{\'e}rale de l'int{\'e}gration et
  sur la classification des transcendantes.
\newblock In {\em Annales scientifiques de l'{\'E}cole Normale Sup{\'e}rieure},
  volume~15, pages 243--384, 1898.

\bibitem{drach1915irreductibilite}
Jules Drach.
\newblock Sur le groupe de rationalit\'e des \'equations du second ordre de m.
  painlev\'e.
\newblock {\em Bull. Sci. Math.}, 39:149--166, 1915.

\bibitem{iwasaki2008finite}
Katsunori Iwasaki.
\newblock Finite branch solutions to {Painlev{\'e}} {VI} around a fixed
  singular point.
\newblock {\em Advances in Mathematics}, 217(5):1889--1934, 2008.

\bibitem{iwasaki2013gauss}
Katsunori Iwasaki, Hironobu Kimura, Shun Shimemura, and Masaaki Yoshida.
\newblock {\em From {Gauss} to {Painlev{\'e}}: a modern theory of special
  functions}, volume~16.
\newblock Springer Science \& Business Media, 2013.

\bibitem{kiso1979local}
Kazuhiro Kiso.
\newblock Local properties of intransitive infinite {Lie} algebra sheaves.
\newblock {\em Japanese journal of mathematics. New series}, 5(1):101--155,
  1979.

\bibitem{kolchin1973}
Ellis Kolchin.
\newblock {\em Differential algebra and algebraic groups}.
\newblock Academic press, 1973.

\bibitem{landesman2008generalized}
Peter Landesman.
\newblock Generalized differential {Galois} theory.
\newblock {\em Transactions of the American Mathematical Society},
  360(8):4441--4495, 2008.

\bibitem{lisovyy2014algebraic}
Oleg Lisovyy and Yuriy Tykhyy.
\newblock Algebraic solutions of the sixth {Painlev{\'e}} equation.
\newblock {\em Journal of Geometry and Physics}, 85:124--163, 2014.

\bibitem{mackenzie1987lie}
Kirill.~C Mackenzie.
\newblock {\em Lie groupoids and Lie algebroids in differential geometry},
  volume 124.
\newblock Cambridge university press, 1987.

\bibitem{mackenzie2005general}
Kirill~C Mackenzie.
\newblock {\em General theory of {Lie} groupoids and {Lie} algebroids}, volume
  213.
\newblock Cambridge University Press, 2005.

\bibitem{malgrange2001}
Bernard Malgrange.
\newblock Le groupo\"ide de {G}alois d'un feuilletage.
\newblock {\em L'enseignement math\'ematique}, 38(2):465--501, 2001.

\bibitem{le2010algebraic}
Bernard Malgrange.
\newblock Differential algebraic groups.
\newblock In {\em Algebraic Approach to Differential Equations (Ed. L{\^e},
  D{\~u}ng Tr{\'a}ng)}, pages 292--312. World Scientific, 2010.

\bibitem{malgrange2010pseudogroupes}
Bernard Malgrange.
\newblock {\em Pseudogroupes de {Lie} et th{\'e}orie de {Galois}
  diff{\'e}rentielle}.
\newblock IHES, 2010.

\bibitem{morimoto1977intransitive}
Tohru Morimoto.
\newblock On the intransitive {Lie} algebras whose transitive parts are
  infinite and primitive.
\newblock {\em Journal of the Mathematical Society of Japan}, 29(1):35--65,
  1977.

\bibitem{noumi2002new}
Masatoshi Noumi and Yasuhiko Yamada.
\newblock A new {Lax} pair for the sixth {Painlev{\'e}} equation associated
  with.
\newblock In {\em Microlocal analysis and complex Fourier analysis}, pages
  238--252. World Scientific, 2002.

\bibitem{okamoto1980polynomial}
Kazuo Okamoto et~al.
\newblock Polynomial hamiltonians associated with {Painlev{\'e}} equations, i.
\newblock {\em Proceedings of the Japan Academy, Series A, Mathematical
  Sciences}, 56(6):264--268, 1980.

\bibitem{painleve1902irreductibilite}
Paul Painlev\'e.
\newblock D\'emonstration de l'irr\'eductibilit\'e absolue de l'\'equation
  $y_{xx} = 6y^2 +x$.
\newblock {\em C.R. Acad. Sci. Paris}, 135:641--647, 1902.

\bibitem{picard1887equations}
{\'E}mile Picard.
\newblock Sur les {\'e}quations diff{\'e}rentielles lin{\'e}aires et les
  groupes alg{\'e}briques de transformations.
\newblock In {\em Annales de la Facult{\'e} des sciences de Toulouse:
  Math{\'e}matiques}, volume~1, pages A1--A15, 1887.

\bibitem{sanz2018differential}
David~Bl{\'a}zquez Sanz, Guy Casale, and Juan Sebasti{\'a}n~D{\'\i}az Arboleda.
\newblock Differential {Galois} theory and isomonodromic deformations.
\newblock {\em arXiv preprint arXiv:1810.08566}, 2018.

\bibitem{umemura1996differential}
Hiroshi Umemura.
\newblock Differential {Galois} theory of infinite dimension.
\newblock {\em Nagoya Mathematical Journal}, 144:59--135, 1996.

\bibitem{vessiot1904theorie}
Ernest Vessiot.
\newblock Sur la th{\'e}orie de {Galois} et ses diverses
  g{\'e}n{\'e}ralisations.
\newblock In {\em Annales scientifiques de l'{\'E}cole Normale Sup{\'e}rieure},
  volume~21, pages 9--85, 1904.

\bibitem{vessiot1946generale}
Ernest Vessiot.
\newblock Sur une th\'eorie g\'en\'erale de la r\'eductibilit\'e des
  \'equations et syst\`emes d'\'equations finies ou diff\'erentielles.
\newblock {\em Ann. Sci. \'{E}cole Normale Sup.}, 63, 1946.

\end{thebibliography}
\end{document}